\documentclass[ 12pt]
{amsart}

\usepackage{enumerate}
\usepackage{ifpdf}
\usepackage{amsmath}
\usepackage{amsfonts}
\usepackage{amssymb}
\usepackage{amsthm}
\usepackage{setspace}
\usepackage{amsrefs}
\usepackage[usenames]{xcolor}

\usepackage{hyperref}
\hypersetup{hidelinks}

\title[]{The $\dbar$-equation on variable strictly pseudoconvex domains}
\author[]{Xianghong Gong}

 \address{Department of Mathematics,
 University of Wisconsin, Madison, WI 53706, U.S.A.}
 \email{gong@math.wisc.edu}

\author[]{Kang-Tae Kim}
\address{Department of Mathematics, POSTECH, Pohang 790-784, The
Republic of Korea}
\email{kimkt@postech.ac.kr}


\newcommand{\dist}{\operatorname{dist}}

\newtheorem{thm}{Theorem}[section]
\newtheorem{cor}[thm]{Corollary}
\newtheorem{prop}[thm]{Proposition}
\newtheorem{lemma}[thm]{Lemma}

\newcommand{\fl}[1]{\lfloor #1\rfloor}

\theoremstyle{definition}

\newtheorem{defn}[thm]{Definition}
\newtheorem{exmp}[thm]{Example}

\newtheorem{rem}[thm]{Remark}

\renewcommand{\th}[1]{\begin{thm}\label{#1}}
\renewcommand{\eth}{\end{thm}}
\newcommand{\co}[1]{\begin{cor}\label{#1}}
\newcommand{\eco}{\end{cor}}
\renewcommand{\le}[1]{\begin{lemma}\label{#1}}
\newcommand{\ele}{\end{lemma}}
\newcommand{\pr}[1]{\begin{prop}\label{#1}}
\newcommand{\epr}{\end{prop}}

\newcommand{\ga}{\begin{gather}}
\newcommand{\ega}{\end{gather}}
\newcommand{\gan}{\begin{gather*}}
\newcommand{\egan}{\end{gather*}}
\newcommand{\al}{\begin{align}}
\newcommand{\eal}{\end{align}}
\newcommand{\aln}{\begin{align*}}
\newcommand{\ealn}{\end{align*}}
\newcommand{\eq}[1]{\begin{equation}\label{#1}}
\newcommand{\eeq}{\end{equation}}

\newcommand{\DD}[2]{\frac{\partial #1}{\partial #2}}
\newcommand{\f}[2]{\frac{#1}{#2}}

\newcommand{\ci}{~\cite}

\newcommand{\cc}{{\bf C}}
\newcommand{\nn}{{\bf N}}
\newcommand{\zz}{{\bf Z}}
\newcommand{\rr}{{\bf R}}

\newcommand{\ov}{\overline}

\newcommand{\id}{\operatorname{I}}

\newcommand{\RE}{\operatorname{Re}}

\renewcommand{\dbar}{\overline\partial}

\newcommand{\cL}{\mathcal}

\newcommand{\all}{\alpha}

\newcommand{\gaa}{\gamma}
\newcommand{\Gaa}{\Gamma}
\newcommand{\del}{\delta}

\newcommand{\var}{\varphi}
\newcommand{\e}{\epsilon}
\newcommand{\om}{\omega}
\newcommand{\Om}{\Omega}

\newcommand{\la}{\lambda}

\newcommand{\pd}{\partial}

\newcommand{\re}[1]{(\ref{#1})}
\newcommand{\rea}[1]{$(\ref{#1})$}
\newcommand{\rl}[1]{Lemma~\ref{#1}}

\newcommand{\rp}[1]{Proposition~\ref{#1}}
\newcommand{\rt}[1]{Theorem~\ref{#1}}
\newcommand{\rd}[1]{Definition~\ref{#1}}
\newcommand{\rrem}[1]{Remark~\ref{#1}}

\newcommand{\rla}[1]{Lemma~$\ref{#1}$}

\newcommand{\rpa}[1]{Proposition~$\ref{#1}$}
\newcommand{\rta}[1]{Theorem~$\ref{#1}$}

\newcommand{\supp}{\operatorname{supp}}

\newcommand{\db}{\dbar}

\newcommand{\sps}{strictly plurisubharmonic}
\newcommand{\psh}{plurisubharmonic}

\newcommand{\spc}{strictly pseudoconvex}

\newcounter{pp}
\newcommand{\bpp}{\begin{list}{$\hspace{-1em}\alph{pp})$}{\usecounter{pp}}}
\newcommand{\epp}{\end{list}}

\newcounter{ppp}
\newcommand{\bppp}{\begin{list}{$\hspace{-1em}(\roman{ppp})$}{\usecounter{ppp}}}
\newcommand{\eppp}{\end{list}}

\def\beq{\begin{equation}}
\def\eeq{\end{equation}}

\begin{document}
\begin{abstract}
We investigate regularity properties of the $\db$-equation on  domains in a complex euclidean space that
depend on a parameter.  Both the interior regularity and the regularity in the parameter are obtained
for a continuous family of   pseudoconvex domains. The boundary regularity
and the regularity in the parameter are also obtained for smoothly bounded strongly pseudoconvex
domains.
\end{abstract}

 \maketitle


\setcounter{thm}{0}\setcounter{equation}{0}
\section{Introduction}\label{sec1}

We are concerned with regularity properties of  the solutions of the $\db$-equation on the domains $D^t$ that
depend on a parameter.
We assume  that   the {\it total space}
$\cL D:=\cup_{t\in[0,1]}D^t\times\{t\}$ is an open subset of $\cc^{n}\times[0,1]$. Such a family
 $\{D^t\colon t\in[0,1]\}$  is called a  {\it continuous  family} of   domains $D^t$ in $\cc^n$,
 or  {\it variable domains} for brevity.  Throughout the paper, the parameter $t$ has the range
 $[0,1]$ unless specified otherwise.

 Let us first introduce   H\"older spaces for variable domains.
  Let $0\leq\all<1$.
 Let $\nn=\{0,1,\dots\}$ and $\ov\nn=\nn\cup\{\infty\}$.
 A family $\{f^t\}$ of functions $f^t$ on $D^t$
is said to be in $C^{\all,0}(\cL D)$, if the function  $(x,t)\to f^t(x)$ is continuous on $\cL D$ and
it has finite $\all$-H\"older norms in $x$ variables   on   compact subsets of $\cL D$ (see \rd{cbj}).   By $\{f^t\}\in
C^{\ell+\all,j}_*(\cL D)$ for $\ell,j\in\ov\nn$,  we mean that all partial derivatives $\pd_x^L\pd_t^i f^t(x)$ are in $C^{\all,0}(\cL D)$ for  $|L|\leq \ell$ and $i\leq j$.
For $k,j\in\ov\nn$ with $k\geq j$,
  let $C^{k+\all,j}(\cL D)$ denote the intersection of all $C_*^{\ell+\all,i}(\cL D)$ with $i\leq j,\ell+i\leq k$.
Analogously, a family $\{f^t\}$ of $(0,q)$-forms $f^t$ on $D^t$
is   in $\cL C^{k+\all,j}(\cL D)$
 if $\{f^t_I\colon t\in[0,1]\}$ are
in the space,
where coefficients
 $f^t_I$ are defined in
 $
 f^t=\sum f_{i_1\dots i_q}^t(z)\, d\ov z_{i_1}\wedge\dots\wedge d\ov z_{i_q}
 $ with $i_1<i_2<\dots<i_q$.

 We say that a {\it smooth} family $\{D^t\}$ of
 bounded domains $D^t$ has {\it $C^{k+\all,j}$ boundary}, if $D^t$ admit defining functions $r^t$ on $U^t$ (with $\nabla_x r^t\neq 0$ at $x\in\pd D^t$) such that $\{r^t\}\in C^{k+\all,j}(\cL U)$,
where $\{U^t\}$ is a continuous family of domains of which the total space $\cL U$ is bounded and contains $\ov{\cL D}$. Finally, a family $\{f^t\}$ of $(0,q)$-forms  $f^t$ with coefficients $f_I$ defined on $\ov{D^t}$ is of class $C^{k+\all,j}(\ov{\cL D})$, if    all partial derivatives $\pd_x^L\pd_t^i f_I^t(x)$, defined on $\cL D$, extend  continuously to $\ov{\cL D}$ and their $\all$-H\"older norms  on $\ov{\cL D}$ in $x$ variables are finite.


We will prove the following  boundary and interior   regularity results.
\th{dbartb} Let   $1\leq q\leq n$.  Let $k,\ell,j\in\ov\nn$ satisfy $k\geq j$ and let $0<\all<1$. Let $\{D^t\}$ be a continuous family of non-empty domains in $\cc^n$ with an open total space    $\cL D$
  in $\cc^n\times[0,1]$.
Let  $\{f^t\}$ be a   family of
    $\db$-closed  $(0,q)$-forms $f^t$ on $D^t$. 
\bppp
\item 
Assume
 that   $\{D^t\}$ is a family of bounded domains of
  $C^{k+2,j}$ boundary,  $\{f^t\}\in C^{k+1,j}(\ov{\cL D})$, and each  $D^t$ is
    strongly pseudoconvex.
 There exist  solutions $u^t$ to $\db u^t=f^t$ on $D^t$
satisfying $\{u^t\}\in C^{k+3/2,j}(\ov{\cL D})$.
\item
Assume that domains $D^t$ respectively admit plurisubharmonic exhaustion functions $\var^t$  with   $\{\var^t\}\in C^{0,0}(\cL D)$ such that   $\{(z,t)\in\cL D\colon\var^t(z)\leq c\}$ is compact for each $c\in\rr$.
Assume that  $\{f^t\}\in C_*^{\ell+\all,j}(\cL D)$ $($resp.\!\! $C^{k+\all,j}(\cL D))$.
There exist 
solutions $u^t$   to $\db u^t=f^t$ on $D^t$ so that $\{u^t\}\in  C_*^{\ell+1+\all,j}(\cL D)$
 $($resp.\!\! $C^{k+1+\all,j}(\cL D))$.
\eppp
\eth
   We will call the above $\{\var^t\}$ a family of  plurisubharmonic {\it uniform}  exhaustion functions   of $\{D^t\}$.
When $n=1$, a precise boundary regularity is given by \rt{spc1}.
Note that $\{D^t\}$ in ($i$) satisfies the conditions in $(ii)$. Another example for $(ii)$ is the following $\cL D$ with  rough boundary.
\begin{exmp}
Let  $D^t$ be the ball in $\cc^n$  with radius $t^{-1}$ and
center  ${\mathbf c}(t)$. Let $D^0=\cc^n$.  When $\mathbf c$ is continuous in $t\in(0,1]$ and $t^{-1}-|{\mathbf c}(t)|\to+\infty$ as $t\to0^+$,
the total space of $\{D^t\}$ is open in $\cc^n\times[0,1]$, while
$$
\var^t(z)=|z|^2+\f{t^2}{1-t^2|z-{\mathbf c}(t)|^2}
$$
 are plurisubharmonic uniform exhaustion functions
  on  $D^t$ satisfying  $\{\var^t\}\in C^{0,0}(\cL D)$.
\end{exmp}


%
%

The study of  regularity  of solutions of the $\db$-equation for a fixed domain   has a long history.
Let us recall some related results.
 The existence of the   smooth   solutions on  a  Stein manifold
follows from    Dolbeault's  theorem and Cartan's Theorem B, as 
 observed
by Dolbeault~\cite{Do56}. It is also a main result of the
 $L^2$-theory (see H\"ormander~\cite{Ho65},~\cite{Ho90}).  The existence and
$C^\infty$  boundary regularity of the canonical
solutions for a strongly pseudoconvex compact manifold  with smooth boundary
were established by Kohn~\cite{Ko63} via investigating  the $\db$-Neumann problem;
Kohn \cite{Ko73} also obtained    $C^\infty$ boundary  regularity of
possibly  non-canonical solutions
for a smoothly bounded
 pseudoconvex domain in $\cc^n$
  (for instance, see  Chen-Shaw~\cite{CS01}*{p.~122}). The exact regularity in H\"older spaces of $\dbar$ solutions for $(0,q)$-forms was obtained by Siu~\ci{Si74}
 for $q=1$  and by Lieb-Range~\ci{LR80} for $q\geq1$.

  The domain dependence of the $\db$-equation has however attracted  less attention.
The $C^\infty$ regularity of solutions for elliptic partial differential equations on a family of
compact complex
manifolds (without boundary) was obtained  by  Kodaira and Spencer~\cite{KS60},
which plays an   important role  in their deformation theory.
For  planar domains depending on a parameter, the exact regularities of Dirichlet and Neumann problems were obtained   recently by Bertrand and Gong~\cite{BG14}.
Notice that
solving the $\dbar$-equation that depends on a parameter
has   played  an important role in the construction of the Henkin-Ram\'irez  functions.
However,  the domain   in this situation is    fixed, 
while    multi-parameters enter into  the non-homogenous $\db$-equation.  Such parameter
dependence is
easy to
understand once a linear $\db$-solution operator is constructed. Of course, the construction
of such a linear solution operator is included in   the classical homotopy formulae;
see~\cite{Ra69},  \cite{GL70},  \cite{He70},  \cite{Li70},   and~\cite{Ra86}.
An interesting case  is the  stability of the $\dbar$-equation in terms of a family
of strongly pseudoconvex
domains; see 
 Greene-Krantz~\cite{GK82}.
Their stability results for the $\db$-solutions can be characterized as the
continuous   dependence in parameter
   as   defined in   section~\ref{sect:holdsp}.  In~\cite{DT91},
Diederich-Ohsawa obtained $C^\infty$ regularity of
  canonical solutions of the $\db$-equation for certain  smooth $(n,1)$-forms. They proved
  the results via
 H\"ormander's $L^2$ technique for a   family of domains in a K\"ahler manifold.

Our approach relies essentially on  solution operators of the $\db$-equation that are represented by  integral formulae  for a  smoothly bounded strictly pseudoconvex domain.
To deal with variable domains, we will use the Grauert bumps to extend $\db$-closed forms to
 a continuous family of larger domains, keeping the forms $\db$-closed. For a fixed domain, the extension technique is well known through
the works of Kerzman~\cite{Ke71} and others. To apply the extension
for a continuous family of strongly pseudoconvex domains, we will  obtain  precise regularity results first for a smooth family of strictly convex domains
by using the Lieb-Range solution operator~\cite{LR80}.  The extension
  allows  us to freeze the domains to apply the classical integral $\dbar$-solution operators (\cite{GL70}, ~\cite{He70},~\cite{Ov71}). By using a    partition of unity in parameter $t$, we will   solve the $\db$-equation  for variable domains with the desired regularity.
 However,  in order to freeze the domains we must   restrict them in $\cc^n$.
    Therefore,  there are several questions arising from our approach. For instance,  it would
be interesting to know if  a more general result  can be established for  the $\db$-equation on
a family of Stein manifolds.
We notice   a remarkable construction  of an integral $\db$-solution operator by Michel~\cite{Mi91}
for a smoothly bounded {weakly}
 pseudoconvex domain in $\cc^n$.  It would be interesting to know
  if the assertion in  \rt{dbartb} $(ii)$ remains true when the given  domains  are only
weakly   pseudoconvex. 

The paper is organized as follows.

In section~\ref{sect:holdsp},  we define H\"older spaces for functions on
 variable domains.
In section~\ref{sect:bumps}, we  adapt the
Narasimhan lemma and Grauert bumps for variable domains.
In section~\ref{sect:breg},  we study the
 regularity of $\db$-solutions on variable domains first for strictly convex case
 and then for
strictly pseudoconvex case.
 The   Lieb-Range solution operator and Kerzman's extension method~\cite{Ke71} for $\db$-closed forms are used in section~\ref{sect:breg} where    \rt{dbartb} ($i$) is proved in \rt{spc}.

In section~\ref{sect:hfp}, we obtain Henkin-Ram\'irez functions for  strictly  pseudoconvex  open sets
that depend on a parameter, which in turn gives us a
   homotopy
formula  for variable strictly pseudoconvex domains. The Henkin-Ram\'irez functions
are used in
section~\ref{sect:ireg} to obtain a parametrized version of the Oka-Weil approximation.   \rt{dbartb} $(ii)$
 is proved in \rt{intr}. As an application,
we  solve a parametrized   version of the Levi problem for variable domains in $\cc^n$.
Finally, we use the $\db$-solutions with parameter to solve a parametrized version of Cousin
problems.

\setcounter{thm}{0}\setcounter{equation}{0}
\section{H\"older spaces for functions on variable domains}
\label{sect:holdsp}

 We first describe some notation used for the rest of the exposition.
We use real variables $x=(x_1,\ldots, x_{d})$ for $\rr^d$. In our application, $d=2n$.
Let $\pd_{x}^{k}$ denote the set of partial derivatives in $x$ of
order $k$.
Let $\hat\pd_x^k$ denote the set of partial derivatives in $x$ of order $\leq k$.
   Recall that $\nn=\{0,1,2,\dots\}, \ov\nn=\nn\cup\{\infty\}$. Set $\rr_+=[0,\infty)$ and $\ov\rr_+=\rr_+\cup\{\infty\}$.

The main purpose of this section is to define H\"older spaces for functions on variable domains.
 When proving boundary regularities of $\db$-solutions, we need to parameterize the variable domains up to boundary.
Therefore, we will define these spaces first via parametrization. We will then define the spaces without using parametrization. We will discuss the relation between two definitions.

When $a$ is a real number, we denote by   $\fl{a}$ 
 the largest integer that does not
exceed $a$.
 Let $D\subset\rr^{d}$ be a bounded domain  with $C^1$ boundary.
Let $  C^{a}(\ov D)$
 be the standard H\"older space  with  norm
$|\cdot|_{D;a}$ on $\ov D$.  Let $j\geq0$ be an integer.
Let $\{u^ t\colon t_0\leq t\leq t_1\}$, with $t_0,t_1$  finite, be a family of functions $u^t$ on $\ov D$. We say that
  it belongs to $C_*^{a, j}(\ov D\times[t_0,t_1])$, abbreviated by $\{u^ t\}
  \in C_*^{a, j}(\ov D\times[t_0,t_1])$,
    if
$ t\mapsto\pd_t^iu^t$, with $ i\leq j$,
are continuous maps from
$[t_0,t_1]$   into $ C^{\fl{a}}(\ov D)$,
and if they are bounded maps sending $[t_0,t_1]$ into $ C^{a-\fl{a}}(\ov D)$.
For a real number $a$ and an integer $j$ with $a\geq j\geq 0$, define
$$
C^{a,j}(\ov D\times[t_0,t_1])
=\bigcap_{i\leq j}C^{a-i,i}_*(\ov D\times[t_0,t_1]).
$$
 For brevity, we write
$C_*^{b,j}(\ov D)=C_*^{b,j}(\ov D\times[0,1])$ and $C^{a,j}(\ov D)=C^{a,j}(\ov D\times[0,1])$.

We now define H\"older spaces on variable domains given by a parametrization. 
Let $D^ t$ be   domains
 in $\rr^d$ and let $\Gamma^ t$  be $ C^1$ embeddings from
 $\ov D$  onto $\ov {D^ t}$.
Let $\{u^t\}$ be a family of functions $u^t$   respectively defined on $D^t$. Write
 $\{u^ t\}
 \in{ C}_*^{a,j}(\ov D_\Gamma)$
when  $\{v^ t\circ\Gamma^ t\}$ is in $ { C}_*^{a,j}(\ov{ D})$.
Define
\gan
 C^{a,j}(\ov D_\Gamma)= \bigcap_{i=0}^j{ C}_*^{a-i, i}(\ov D_\Gamma),
\end{gather*}
for   an integer $j$ in $[0,a]$.   We define $C_*^{\infty, j}(\ov D_\Gamma)=\bigcap_{k=1}^\infty C_*^{k, j}(\ov D_\Gamma)$.
Similarly, define
$C^{\infty, j}(\ov D_\Gamma)$, $C_*^{\infty, \infty}(\ov D_\Gamma)$, and
$C^{\infty, \infty}(\ov D_\Gamma)$.

While writing $\{u^ t\}$ as $u$ and $\{v^t\circ\Gamma^ t\}$ as $v\circ\Gaa$,
we define  the  norms:
\begin{gather*}
|u|_{D;a,i }=\sup_{0\leq \ell\leq i,   t\in[0,1]}
\{|\pd_ t^\ell u^ t |_{D;a}\}\quad \text{if $u\in C_*^{a,i}(\ov D)$};\\
  \|u\|_{D;a,j}=\max_{0\leq i\leq j} \{|u|_{D;a-i, i}\}, \quad \text{if $u\in C^{a,i}(\ov D)$},\\
 |v|_{D_\Gaa;a,j }=|v\circ\Gaa|_{D;a,j},\quad \text{if $v\in C_*^{a,j}(\ov D_\Gaa)$};\\
  \quad  \|v\|_{D_\Gaa;a,j }  =\|v\circ\Gaa\|_{D;a,j}, \quad \text{if $v\in C^{a,j}(\ov D_\Gaa)$}.
\end{gather*}
%
 Note that when defining $C^{a,j}(\ov D_\Gaa)$, we assume that $D$ is bounded with $C^1$ boundary and $a\geq j$.  Let us first ensure   the independence of
  the H\"older spaces on the parametrization under mild conditions.
\begin{lemma}\label{chainle} 
Let $D, D^t$ be  bounded domains in $\rr^d$ with $C^1$ boundary, and let
 $\{\Gaa^t\}\in C^{a,j}(\ov { D})\cap
 C^{1,0}(\ov { D})$  be a family of embeddings $\Gaa^t$  from $ \ov{D}$ onto $\ov{D^t}$.
 Then we have the following.
\bppp
\item $\cL D=\cup_t D^t\times\{t\}$ is open in $\rr^d\times[0,1]$,
 $\ov{\cL D}$ is compact, and $\ov{\cL D}=\cup_t \ov{D^t}\times\{t\}$.
\item A family $u=\{u^t\}$ of functions $u^t$
 on $\ov{D^t}$ is in $C^{a,j}(\ov {D}_{\Gaa})$ if and only if
\gan
\pd_t^\ell\pd_x^Ku^t(x)\in C^0(\ov{\cL D}), \quad \forall \ell\leq j,   |K|\leq a-\ell;\\
\|u\|_{\cL D;a,j}:= \max_{0\leq i\leq j} \{|u|_{\cL D;a-i, i}\}<\infty
\end{gather*}
with $
|u|_{\cL D;b,i }:=\sup_{0\leq \ell\leq i,   t\in[0,1]}
\left\{|\pd_ t^\ell u^ t |_{{D^t};b}\right\}.
$ 
For some  constants $C_1,C_2$ depending on $ \|\{\Gaa^t\}\|_{D;a,j}$ and
$\inf_{\cL D} |\pd_x\Gaa^t|$
$$
C_1^{-1}\|u\|_{D_\Gaa;a,j}\leq
\|u\|_{\cL D;a,j}\leq C_2\|u\|_{D_\Gaa;a,j}.
$$
\eppp
\end{lemma}
\begin{proof} We get ($i$) easily, since
$
(x,t)\mapsto (\Gaa^t(x),t)
$
defines a   homeomorphism $\Gaa$ sending $ \ov D\times[0,1]$ onto   $\cup_t\ov{D^t}\times\{t\}$.
Thus $\cL D=\Gaa(D\times[0,1])$ is open in $\rr^d\times[0,1]$ and $\pd\cL D:=\Gaa(\pd D\times[0,1])$
is compact.

Throughout  the paper,
the boundary value of a partial derivative  $\pd_t^\ell\pd_x^Ku^t$  at a point in $\pd{\cL D}$
is regarded as an extension in the pointwise limit, if it exists,
of the derivatives in the open set $\cL D$.

We now verify $(ii)$. Since $\Gaa^t$ are embeddings and $\{\Gaa^t\}\in
 C^{1,0}(\ov { D})$, the Jacobian matrix $\pd_x \Gaa^t$ is non-singular and continuous on $\ov D\times[0,1]$. Since $\pd D\in C^1$,
   the fundamental theorem of calculus yields
 $$
 |(\pd_x\Gaa^t)(x'-x)|/2\leq |\Gaa^t(x')-\Gaa^t(x)|\leq 2 |(\pd_x\Gaa^t)(x'-x)|,
 $$
 when $x'$ is sufficiently close to $x$.
This shows that
$$
 |x'-x|/C\leq |\Gaa^t(x')- \Gaa^t(x)|\leq C |x'-x|.
 $$
Thus, we obtain the lemma for $a<1$. Let $y=\Gaa^t(x)$.  By the chain rule, we have
$$
\pd_y((\Gaa^t)^{-1})=(\pd_x\Gaa^t)^{-1}\circ(\Gaa^t)^{-1}(y), \
\pd_t(\Gaa^t)^{-1}(y)=-((\pd_x\Gaa^t)^{-1}\pd_t\Gaa^t)\circ(\Gaa^t)^{-1}(y).
$$
Assume that $\{u^t\}\in C^{1,0}$. Since $\{\Gaa^t\}\in C^{1,0}$,   the chain rule says that
$$
(\pd_yu^t)\circ\Gaa^t\pd_x\Gaa_1^t=\pd_x(u^t\circ\Gaa^t(x)).
$$
This shows that $\{u^t\}$, $\{\pd_yu^t\}$ are in $ C^{0,0}$ if and only if $\{u^t\}\in C^{1,0}$. Assume that $\{u^t\}\in C^{1,1}$. Using the existence of partial derivatives $\pd_yu^t$ and computing limits directly, we verify that
$$
(\pd_tu^t)\circ\Gaa ^t(x)=\pd_t(u^t\circ\Gaa^t(x))-(\pd_yu^t)\circ\Gaa^t(x)\pd_t\Gaa ^t(x)
$$
for $x\in D$ and hence for all $x\in\ov{ D}$ by continuity.  This shows that $\{u^t\}$, $\{\pd_yu^t\}$, and $\{\pd_tu^t\}$ are in $C^{0,0}$ if and only if $\{u^t\}$ is in $C^{1,1}$.
The lemma follows immediately for other values of $a,j$.
\end{proof}

We now define $C^{a,j}_*(D_\Gaa)$ for a family of $C^1$ embeddings  $\Gaa^t$ from
an arbitrary open set $D\subset\rr^d$ onto $D^t\subset\cc^n$.
By $\{f^t\}\in C_*^{a,j}(D_\Gaa)$, we  mean that $f^t$ are functions
on $D^t$ such that
$\{f^t\}\in   C_*^{a,j}(K_\Gaa)$ for any subset  $K$ of  $D$
 with smooth
boundary.  Again, when all $\Gaa^t$ are the identity map, we define $C_*^{a,j}(D)$ to
be  $C_*^{a,j}(D_\Gaa)$.
Define $C^{a,j}(D_\Gaa)$ and $C^{a,j}(D)$ similarly.


 We will denote by $\cL U(E)$ a neighborhood of $E$ when $E$ is a subset of $\rr^d$.
For example, we say that
 $D$ is defined by $r<0$ if $r$ is a real function
on some $\cL U(\ov D)$
on which $D$ is defined by $r<0$.


We will use the following Seeley extension operator~\ci{Se64}.
\le{seeleyext} Let $H=\rr^d\times[0,\infty)$.
There is a continuous linear extension operator
$E\colon C_0^0(H)
\to C_0^0 (\rr^{d+1})$
 such that $Ef=f$ on $H$ and
 $E: C_0^a(H)
\to C_0^a (\rr^{n+1})$ is continuous  for each $a\geq0$.
\ele
Here $C^{\bullet}_0$ stands for functions with compact support.
Seeley~\cite{Se64} showed that there
 are  numerical sequences $\{a_k\}, \{b_k\}$   such that ($i$) $b_k < 0$ and $b_k\to-\infty$, $(-1)^ka_k>0$,
$(ii)$ $\sum |a_k|\cdot |b_k|^n<\infty$  for $n = 0, 1, 2, \ldots$, $(iii)$ $\sum a_k(b_k)^n=1$  for
$n = 0, 1, 2, \ldots$.   Then Seeley defined the extension
$$
(Ef)(x, s)=\sum_{k=0}^\infty a_k\phi(b_ks)f(x, b_ks).
$$
Here $\phi$ is a $C^\infty$ function satisfying $\phi(s)=1$ for $s<1$ and $\phi(s)=0$
for $s>2$. For a differential form $f$, we define $Ef$ by extending the coefficients of $f$ via $E$.

Let us use the extension operator to discuss the space $C^{a,j}(\ov D_\Gaa)$
and a  version of   extension  for variable domains. We will also discuss an approximation.
 Let $C_0^{a,j}(D_\Gaa)$ denote the set of   $\{f^t\}\in C^{a,j}(D_\Gaa)$ such that $(x,t)\to f^t (x)$ has compact support in the total space of
 $\{\Gaa^t(D)\}$.
\le{seeleyp}  Let $D$ be a bounded domain
in $\rr^d$ with $C^a$ boundary and $a\geq1$.  Let $\{\Gaa^t\}\in C^{a,j}(\ov D)$ be a family of embeddings $\Gaa^t$ from $\ov D$ onto $\ov{D^t}$.
\bppp
\item
 There exist an open neighborhood $U$ of $\ov D$  with $\pd  U\in C^a$
 and a family of embeddings $\hat\Gaa^s$ from $\ov U$ onto $ \ov{U^s}$ for $s_0\leq s\leq s_1$
 with  $s_0<0$, $1<s_1$ such that $\{\hat\Gaa^s\colon s_0\leq s\leq s_1\}$ is of class $C^{a,j}(\ov U)$. Moreover, $\hat\Gaa^t=\Gaa^t$ on $\ov D$ for $t\in[0,1]$.
There exists an $\rr$-linear  extension operator
$$
E\colon C^{a,j}(\ov D_\Gaa)\to C^{a,j}_0(U_{\hat \Gaa})$$
satisfying
$
\|Ef\|_{U_{\hat\Gaa};a',j'}\leq C_{a}\|f\|_{D_\Gaa; a',j'}$  for all finite  $a'\leq a$ and finite $j'\leq j$. Here
$$
(Ef)^t|_{\ov {D^t}}=f^t, \quad f=\{f^t\},\quad\forall t\in[0,1].
$$

\item
Let $f\in C^{b,k}(\ov D_\Gaa)$. There exists
a sequence in $ C^{a,j}(\ov{D}_\Gaa)$ that is bounded in $C^{b,k}(\ov D_\Gaa)$
and converges to $f$ in $C^{\fl{b},k}(\ov D_\Gaa)$.
\eppp
\ele
\begin{rem} In ($i$) of the lemma,
the    $C^{a,j}(\ov U_{\hat\Gaa})$ is defined as
  $C^{a,j}(\ov U_{\Gaa})$,  where
   $\{\Gaa^t\colon 0\leq t\leq 1\}$ is replaced by
    $\{\hat\Gaa^s\colon s_0\leq s\leq s_1\}$.
\end{rem}
\begin{proof} ($i$) Since $\pd D$ is in $C^a$ with $a\geq1$,
we  can  locally use a $C^a$ coordinate change $\var$ to transform
$\pd D$ into the boundary of the half-space $x_{d}\geq0$. We then extend  the mapping
 $f^t:=\Gaa^t\circ\var^{-1}$ to $g^t$ 
  by
\eq{etfx}
g^t (x )
=\sum_{k=0}^\infty a_k\phi(b_kx_{d})f^t(x', b_kx_{d}), \quad\forall x_{d}<0.
\eeq
Set $\Gaa_*^t=g^t\circ\var.$
Thus, using a partition of unity
and local changes of coordinates, we can extend
$\Gaa^t\colon \ov D\to\ov{D^t}$ to  embeddings  $\Gaa_*^t$ from
$\ov{U}$ onto $\ov{ U^t}$ for a smoothly bounded domain $U$  containing
$\ov{D}$, while  $\{\Gaa_*^t\}\in C^{a,j}(\ov{U})$.
Next, we extend $\{\Gaa^t_*\colon 0\leq t\leq 1\}$ to a family
$\{\hat\Gaa^s\colon s_0\leq s\leq s_1\}$.
Let us use
  Seeley's extension for the half-space $s\geq 0$.  By a partition of unity in the $t$
  variable, we may assume
 that $\Gaa_*^t=0$ for $1/2<t<\infty$. Define
\eq{hgsx}
\hat\Gaa^s(x)=\sum_{\ell=0}^\infty a_\ell\phi(b_\ell s)\Gaa_*^{b_\ell s}(x), \quad\forall  s<0.
\eeq
Applying Seeley's extension to the half-space $s\leq 1$,
 we can  extend $\{\Gaa^t\}$ to a family of
 embeddings $\hat\Gaa^s$ from $ \ov{U}$ onto $ \ov{U^s}$ for $s\in [s_0,s_1]$,
 when $-s_0$,  $s_1-1$ are positive and sufficiently small.
We leave it to the reader to check that
$\{\hat\Gaa^s\}$
is in $C^{a,j}(\ov{U})$.  The extension $E\colon C^{a,j}(\ov D_{\Gaa})\to C^{a,j}(\ov{U}_{\hat\Gaa})$
can be defined  by formulas similar to \re{etfx}-\re{hgsx} and by shrinking $\ov U$,
 $[s_0,s_1]$ slightly.

$(ii)$
As above, a family $\{f^t\}\in C^{b,k}(\ov D)$ of functions $f^t$ on $\ov{ D^t}$
 extends to a family $\{\tilde f^t\}\in C^{b,k}(\ov {U}_{\hat\Gaa})$.
Let   $\{f^t\}\in C^{b,k}(\ov{D})$. Applying the standard smoothing operator on  $U\times (s_0,s_1)$
to $\tilde f^s(\hat\Gaa^s(x))$, we can verify the approximation on the compact
subset  $\ov D\times[0,1]$ of $ U\times(s_0,s_1)$.
\end{proof}

 \rl{chainle} says that with $a\geq1$,   space $C^{a,j}(\ov{D_\Gaa})$ does not depend    on   the parameterizations $\Gaa^t\colon \ov D\to\ov{D^t}$, provided they exist.
Next, we study the existence of parameterizations $\Gaa^t$.
To this end, we first introduce function spaces without using parametrization.
 Recall from section~\ref{sec1} that  $\{D^t\colon t\in[0,1]\}$  is a   continuous  family of  domains $D^t$ in $\rr^d$, if the total space
$\cL D=\cup_{t}D^t\times\{t\}$ is open in $\rr^{d}\times[0,1]$.
\begin{defn} \label{cbj} $(i)$ Let $\{D^t\colon t_0\leq t\leq t_1\}$ be a continuous family of domains in $\rr^d$ with total space $\cL D$.  We
 say that a family $\{u^t\}$ of functions $u^t$ on  $\ov{D^t}$  is in
 $C_*^{b,j}(\ov{\cL D})$ for finite $b$ and $j$, if
\gan
\pd_t^\ell\pd_x^Lu^t(x)\in C^0(\ov{\cL D}), \quad \forall \ell\leq j,   |L|\leq b;\\
|u|_{\cL D;b,i }:=\sup_{0\leq \ell\leq i,   t\in[t_0,t_1]}
\left\{|\pd_ t^\ell u^ t |_{{D^t};b}\right\}<\infty.
\end{gather*}
Define $C^{a,j}(\ov{\cL D})=\bigcap_{i\leq j}C_*^{a-i,i}(\ov{\cL D})$ and the norm
$$
\|u\|_{\cL D;a,j}:= \max_{0\leq i\leq j} \{|u|_{\cL D;a-i, i}\}.
$$
 Define $C_*^{b,j}(\ov{\cL D}), C^{a,j}(\ov{\cL D})$ similarly if one of $a,b,j$
is infinity.  Let $C_*^{b,j}(\ov{\cL D})$, $C^{a,j}(\ov{\cL D})$  be  the sets of functions $v$ on $\ov{\cL D}$  with  $\{v(\cdot,t)\}$  in $ C_*^{b,j}(\ov{\cL D})$,  $ C^{a,j}(\ov{\cL D})$, respectively.

$(ii)$ By $\{u^t\}\in C^{a,j}(\cL D)$ we mean that  $\{u^t\}\in C^{a,j}(\ov\om)$
for any relatively  compact open subset $\om$
of $\cL D$ and for $\om^t=\{x\colon(x,t)\in\om\}$.
Define $C_*^{b,j}(\cL D)$,  $ C^{a,j}(\cL D)$,  $C_*^{b,j}({\cL D})$, $C^{a,j}({\cL D})$ analogously.
The topology on $C^{a,j}(\cL D)$ is defined by semi-norms $\|\cdot\|_{\om; a',j'}$, where $a'\leq a, j'\leq j$,
$a',j'$ are finite
 the sets $\om$ are relatively compact in $\cL D$.  Define the topologies of other spaces analogously.
 \end{defn}

For clarity, sometimes we denote  $C^{a,j}(\cL D), C^{a,j}(\ov{\cL D})$ by $C^{a,j}(\{D^t\})$, $C^{a,j}(\{\ov{D^t}\})$, respectively.


%

\pr{nonparem}Let   $a\in\ov\rr_+, j\in\ov\nn$ and $a\geq j$. Let $\{D^t\}$ be a continuous family of non-empty domains in $\rr^d$ with a bounded total space  $\cL D$  and let $b \cL D$ be the boundary of $\cL D$ in $\rr^d\times[0,1]$.  Then the following are equivalent$:$
\bppp
\item For each $t_0\in[0,1]$, there exist a neighborhood $I$ of $t_0$ in $[0,1]$ and
  a family $\{\Gaa^t\}\in C^{a,j}(\ov D\times\ov I)\cap C^{1,0}(\ov D\times\ov I)$ of embeddings
$\Gaa^t$ from $\ov D$ onto $ \ov{D^t}$ with $\pd D\in C^a\cap C^1$.
\item
For every $(x_0,t_0)\in b \cL D$  there exist
an open neighborhood  $\om$ of $(x_0,t_0)$ in $\rr^d\times[0,1]$ and a real function $r\in C^{a,j}(\ov\om)\cap C^{1,0}(\ov\om)$
such that $r(x_0,t_0)=0$, and
\eq{dtr}
 \cL D\cap\om=\{(x,t)\in\om\colon r(x,t)<0\}; \quad \nabla_xr(x,t)\neq0, \quad\forall (x,t)\in b\cL D\cap\om.
\eeq
\item $\ov{\cL D}$ is contained in  a  domain $\om$  in $\rr^d\times[0,1]$   and there is a real function $r$ on $\om$
of class $C^{a,j}\cap C^{1,0}$ such that \rea{dtr} holds.
\eppp
We say that  $ \{ {D^t}\}$ is a smooth family  of   bounded domains  of  $C^{a,j}\cap C^{1,0}$  boundary if its total space is bounded and it satisfies  one of the above equivalent conditions.
\epr
\begin{proof} We first show that ($i$) implies $(iii)$.
Assume that such a family $\{\Gaa_\all^t\}$ exists for $t\in \ov{I_\all}$, where $I_\all$ is a connected open set $I_\all$ in $[0,1]$ and $\{I_\all\}$ is a finite open covering of $[0,1]$. Since $[0,1]$ is connected, we may assume that $D$ is independent of $\all$.
 By \rl{seeleyp} ($i$),  we may assume that
$\Gaa_\all^t$ extend to embeddings $\Gaa_\all^t$ from $\ov{U}$ onto $ \ov{U^t}$ with $\ov {D}\subset U$ and $\{\Gaa_\all^t\}\in C^{a,j}(\ov {U}\times \ov{I_\all})$.
 Let $r_0$ be a real function of class $C^a$ on $U$ such that $D$ is defined by $r_0<0$ and $\nabla r_0\neq0$ on $\pd D$. Then $r(x,t)=\sum\chi_\all(t)r_0\circ(\Gaa_\all^t)^{-1}(x)$
has the desired property, while $\om$ is an open subset of $\bigcup_\all\{(\Gaa_\all^t(x),t)\colon x\in U\}$.

Clearly, $(iii)$ implies $(ii)$. Let us now show that $(ii)$ implies  $(iii)$.
 To this end, let us  show that there are a neighborhood $\om$ of $\ov{\cL D}$ and a real function $r\in\cL C^{a,j}(\ov\om)$ such that
\re{dtr} holds. In other words, $r$ is a global definition function of $\cL D$.
By \re{dtr} and the local function $r$ satisfying $r(x_0,t_0)=0$, it is clear that $\cL D$ does not contain boundary points of $\cL D$. So $\cL D$ is open.
Since $\pd\cL D$ is compact, we cover it by finitely many open sets $\om_\all$ such that on  $\om_\all$
there are real functions $r_\all\in C^{a,j}(\ov{\om_\all})$ such that
$$
\cL D\cap\om_\all\colon r_\all<0;  \quad \nabla_x r_\all(x,t)\neq0, \quad\forall x\in\pd D^t\cap \om_\all.
$$
Let $r_0=-1$ on $\om_0=\cL D$. Choose a partition $\{\chi_0,\chi_\all\}$  of unity by non-negative
functions
  subordinate to
the open covering $\{\om_0,\om_\all\}$. Note that $\nabla_xr_\all(x,t)$ is a non-zero
outward normal vector
at $x\in\pd D^t$.  Shrinking $\om=\om_0\cup(\cup_\all\om_\all)$ slightly, we can verify that $r=\sum\chi_\all r_\all$ has the required property.

Next, let us show that $(iii)$ implies ($i$).
Assume  that the function $r$ is given as in \re{dtr} with $\om$ being an open neighborhood of $\ov{\cL D}$.
Define $\om^{t}=\{x\colon (x,t)\in \om\}$.
 Fix $t_0$. We need to  find a family of embeddings $\Gaa^t$ from $\ov D$ onto $\ov{D^t}$ for $t$ near $t_0$.
 Let $n(x)$ be the  gradient vector field of $r(x,t_0)$. We  approximate $n(x)$ by a vector field $v(x)$ such that
$v$ is of class $C^{a}$ on $\om^{t_0}$ and such that for $x$ in a small neighborhood $V$ of $\pd D^{t_0}$, the line segment $\{x+sv(x)\colon -\e\leq s\leq \e\}$ intersects  $\pd D^{t}$ transversally at
a unique
point with $s=S(x,t)$ for $|t-t_0|<\del$. Note that $s$ is the unique solution to
$$
r(x+sv(x),t)=0, \quad |s|<\e.
$$
Let $D^t_c\subset\om^t$ be defined by $r(\cdot,t)<c$.
Fix a small positive constant $-c_1$.  For $x$ near $\pd D_{c_1}^{t_0}$, let $b\in[-\e,\e]$ be the unique
  number such that $x+bv(x)$  is in $\pd D^{t_0}$.  For $t$ close to $t_0$,
there exists a unique  $\tilde b\in[-\e,\e]$   such that $x+\tilde bv(x)$  is in $\pd D^{t}$.
 Note that $b$ depends on $x$, while $\tilde b$ depends on  $x,t$; and both are positive.
 We will find  $\nu=\nu(x,t,\la)$ that is strictly increasing in $\la$
such that $\nu(x,t,\la)\equiv\la$ for $\la$ near $0$, while  at the end point $\la=b(x)$ we have
$$
 \nu(x,t,b(x))=\tilde b(x,t).
$$
For the existence of $\nu$ and its smoothness, we take a smooth decreasing function $\chi$
such that $\chi(\la)$ equals $1$ near $\la=0$ and  $0$ near $\chi=b$. Furthermore, $\int_0^b\chi\, d\la<b/2$. Note that the latter is less than $\tilde b$ when $t$ is close to $t_0$. Define
\eq{ch1}
\chi_1=\frac{\tilde b(x,t)-\int_0^b\chi\, d\la}{\int_0^b(1-\chi)\, d\la}(1-\chi), \quad \nu=\int_0^\la(\chi+\chi_1)\, d\la.
\eeq
We then define $\Gaa^t=\id$ on $D_{c_1}^{t_0}$ and
$$
 \Gaa^t(x+\la v(x))=x+\nu(x,t,\la)v(x), \quad\text{for $ 0\leq\la\leq b(x),   x\in\pd D^{t_0}_{c_1}$}.
$$
Then  $\Gaa^t$ embeds $  \ov{D^{t_0}}$ onto $\ov {D^t}$ for  $t$ close to $t_0$. To verify the smoothness of $\{\Gaa^t\}$, let   $x,y$ be in $\rr^d$. We start with equations that determine $b=b(x)$:
$$
x+bv(x)=y, \quad r(y,t_0)=0, \quad r(x,t_0)=c_1.
$$
The first two equations determine $b$   via $x$. Indeed, the Jacobian determinant of $y-x-bv(x), r(y,t_0)$
in $y,b$ equals $-v(x)\cdot\nabla_yr(y,t_0)$. The latter is not zero, since $x$ is close to $y$, $ v(x)$ is close to $\nabla_y r(y,t_0)$, and $\nabla_yr(y,t_0)\neq0$ near $\pd D^{t_0}$. This shows that $b$ is a function of class $C^a$ in $x$ near $\pd D^{t_0}$. The $\tilde b=\tilde b(y,t)$ is determined by
$$
x+\tilde bv(x)=y, \quad r(y,t)=0, \quad r(x,t_0)=c_1.
$$
We see that $\tilde b$
 is of class $C^{a,j}$ in $y$    near $\pd D^{t_0}$ and in $t$   near $t_0$.
Finally, we consider equations $\Gaa^t(y)=u(y,t)$, which can be written as
$$
x+\la v(x)=y, \quad r(x,t_0)=c_1, \quad u=x+\nu(x,t,\la)v(x).
$$
We want to use the first two equations to determine $x,\la$ via $y$
 and  the last equation to determine
$u(y,t)$.    Recall that $0\leq\la\leq b(x)\leq\e$ and $\e>0$ is small. At $\la=0$, the Jacobian determinant of $x+\la v(x),r(x,t_0)$ in $x,\la$ is $-|\nabla_xr(x,t_0)|^2$, which is non zero. By the implicit function theorem, we can verify
that $(x,\la)$ is of class $C^a$ in $y$.  By the smoothness properties  of $b,\tilde b$ verified above
and by \re{ch1}, we know that $\nu(x,t,\la)$ is of class $C^{a,j}$ in $y,t$. This shows that $\{\Gaa^t\}$
is of class $C^{a,j}$ as $t$ varies near $t_0$.
\end{proof}

\begin{rem}
By \rl{chainle} and \rp{nonparem},   the $C^{a,j}(\ov D_\Gaa)$  space is independent of
  $\{\Gaa^t\}\in C^{a,j}(\ov D)$  when $a\geq1$.  Furthermore, the parameter $t$ can be in several variables and  a parametrization $\Gaa^t$ can be obtained
for $t$ near a given point $t_0$.
\end{rem}

The smooth approximation for $C^{a,j}(\ov{\cL D})$ is given by \rl{seeleyp}.
We conclude this section with the following approximation result.
\pr{intapp}Let $\{D^t\}$ be a continuous family of domains. Then $C^{\infty,\infty}(\cL D)$ is dense
in $C^{a,j}(\cL D)$ and in $C_*^{b,j}(\cL D)$.
\epr
\begin{proof}
We know that $\cL D$ is open in $\rr^d\times[0,1]$. If $D^0$ is non-empty, we extend $D^t$ to a larger
family by setting $D^t=D^{-t}$ for $-1<t<0$; if $D^1$ is non-empty, we set $D^{2-t}=D^t$ for $0<t< 1$.
Then the total space of the extended family is open in $\rr^d\times (-1,2)$. Using partition of unity and Seeley's extension, we can extend each $\{f^t\}\in C^{b,j}_*(\cL D)$  to a family
$\{\tilde f^t\}$ in $C^{b,j}_*(\{ D^t\colon -1<t<2\})$. By the standard smoothing, we can get the approximation.
\end{proof}

We have provided necessary background for H\"older spaces on variable domains. In our applications,
 boundary regularities  for the $\db$-solutions will be derived only in
$C^{\bullet}(\ov{\cL D})$ spaces,
while the proof of  interior  regularity is more flexible and it will be carried out  in  $C_*^{\bullet}(\cL D)$
 and $C^{\bullet}(\cL D)$.

\setcounter{thm}{0}\setcounter{equation}{0}
\section{Narasimhan lemma and Grauert bumps for
variable domains}
\label{sect:bumps}

The main purpose of this section is to recall a construction of Grauert's bumps. We will provide precise smoothness for  the bumps with a parameter, which are needed for us to understand the boundary regularity of $\db$-solutions on variable domains in section~\ref{sect:breg}.

We need some facts about defining functions of a domain.
A bounded domain $D$ in $\cc^n$ that has $C^2$ boundary is defined by $r<0$,
where $r$ is a $C^2$ defining function defined near $\ov D$ and $\nabla r\neq0$ on $\pd D$.
Then the defining function
$\tilde r =e^{Lr}-1$ enjoys further properties. Assume that $L$ is  sufficiently large.
When $D$ is strictly pseudoconvex,
the complex Hessian of $\tilde r$  is positive definite near $\pd D$.  Note that each connected component of
$D$ is strictly convex if and only if
 the real Hessian $Hr$ is positive definite on the tangent space of $\pd D$.
The latter implies   that $H\tilde r$
is positive definite at each point of $\pd D$.  Finally, $D$ is strictly convex if and only if
$D$ is connected and $H\tilde r$ is positive-definite at each point of $\pd D$, equivalently if
$$
\RE\{\tilde r_\zeta\cdot(\zeta-z)\}\geq |\zeta-z|^2/C, \quad\forall \zeta\in\pd D, z\in\ov D,
$$
for some positive number $C$.  In our proofs, a domain may not be connected, while a convex domain is always connected.

%
\begin{lemma} \label{naragrau}
Let $j\in\ov\nn$ and $a\in\ov\rr_+$ with $j\leq a$.
Let  $\{D^t\} $ be  a smooth family of  bounded domains  in $\cc^n$ with    $C^{a+2,j}$ boundary. Assume that   $D^t$ are   strictly pseudoconvex.
For each $t_0\in[0,1]$,  there are an open neighborhood $I$ of $t_0$, a connected open neighborhood $U$
of $\ov D$ with $\pd U\in C^{a+2}$,   biholomorphic mappings $\psi_i$ from $\om_i$ onto $B_{\e_0}$,  and   smooth families $\{D_i^t\}, \{N_i^t\},\{B_i^t\}$ of bounded domains of   $C^{a+2,j}$ boundaries
 satisfying
the following$\,:$
\bppp
\item The
$\hat N_i^t:=\psi_i( {N_i^t})$ are strictly convex and relatively compact in $B_{\e_0}$, and
\ga\nonumber
D_{i+1}^t= D_i^t\cup B_{i}^t,\quad D_0^t=D^t,\quad \ov {D^t}\subset D_m^t,\quad m<\infty;\\
\label{NiDi}
N_{i}^t=D_i^t\cap B_{i}^t, \quad (\ov {B_{i}^t\setminus D_i^t})\cap (\ov{D_{i}^t\setminus N_i^t})=\emptyset.
\end{gather}
\item There exist a family $\{\Gaa^t\}\in C^{a+2,j}(\ov U)$ of embeddings $\Gaa^t$ from $\ov U$ onto $\ov{U^t}$ and compact subsets $\ov{D_i},\ov{ B_i}, \ov{N_i}$
of the open set $U$ such that
$$
\Gaa^t(\ov{D_i})=\ov{D_i^t}, \quad \Gaa^t(\ov{B_i})=\ov{B_i^t}, \quad \Gaa^t(\ov{N_i})=\ov{N_i^t}.
$$
\item Each
$D_i^t$ {\em (}resp. $\hat N_i^t)$ is  defined by $r_i^t<0$ on $U$ {\em (}resp. $\hat r_i^t<0$ on $B_{\e_0})$, and
 $r^t_i$  is strictly plurisubharmonic near $\ov U\setminus D^t_i$, and $\hat r_i^t$
is strictly convex on ${B_{\e_0}}$. Furthermore, $\{\hat r_i^t\}\in C^{a+2,j}(\ov{B_{\e_0}}\times\ov I)$ and $\{ r_i^t\}\in C^{a+2,j}(\ov U\times\ov I)$.
\eppp
\end{lemma}
\begin{proof}
  We will first construct $N$ and $B$ 
for a fixed domain $D$ by some local changes of coordinates.

 We assume that the defining
function $r$ of $D$  is   $C^{a+2}$ and   strictly plurisubharmonic
near $\pd D$.
Fix a point
$p\in\pd {D}$. We want to construct a bump $B$   containing $p$
and   a biholomorphic mapping $\psi$ defined on $\ov B$ such that $\psi(\ov N)$ is
strictly convex for $N=B\cap D$. Furthermore, $D_1=D\cup B$, $B$,  and $N$
are strictly pseudoconvex with $C^{a+2}$ boundary.  We also
require that
$
 (\ov{B\setminus D})\cap(\ov{D
\setminus N})=\emptyset.
$

More precisely, let us choose a unitary matrix $S$ such that  the map
$
\var_0\colon z\mapsto S(z-p) $
sends the inner normal vector of $\pd D$ at $p$   to the $y_n$-axis. We will apply two more changes of coordinates that are uniquely determined by Taylor
 coefficients of $r_1(z):=\f{1}{2|r_z(p)|}r\circ\var_0^{-1}(z)$ at the origin. We  then specify $B$ and $N$.

Assume that $\var_0$ has been determined.
Near the origin, $D'=\var_0(D)$ is defined by $r_1<0$ with $r_1(z)=- 
y_n+O(2)$.  In
the Taylor polynomial, we have
$$
r_1=-  
y_n+\RE\sum a_{ij} z_iz_j+\sum b_{i\ov j}z_i\ov z_j
+h_1(z)
$$
with $h_1(z)=o(2)$.  Define a coordinate transformation $\tilde z=\var_1(z)$ by
$\tilde z'=z'$ and
 $$
 \tilde z_n=  
 z_n-i\sum a_{jk}z_jz_k-ib_{n\ov n}z_nz_n-  
 i\sum_{\all}b_{\all \ov n}z_\all z_n.
 $$
   Then $D''=\var_1(\cL U(0)\cap D')$ is defined by $r_2<0$ for $r_2:=r_1\circ\var_1^{-1}$. We have
$$
r_2=-y_n+\sum_{\all,\beta=1}^{n-1} b_{\all\ov\beta}z_\all\ov z_\beta+h_2(z)
$$
and $h_2(z)=o(2)$.  Define $\tilde z=\var_2(z)$
by $\tilde z_n=  z_n-   
i  z_n^2$ and $\tilde z'=z'$. Then $D'''=\var_2(\cL U(0)\cap D'')$ is defined by
$r^*<0$ with   $r^*=e^{r_2\circ\var_2^{-1}}-1$ and
$$
r^*=-y_n+  
|z_n|^2+\sum_{\all,\beta=1}^{n-1} b_{\all\ov\beta}z_\all\ov z_\beta+h(z), \quad h(z)=o(2).
$$
Obviously,  
$r^*$ is  $C^{a+2}$
and strictly convex on some $\ov {B_{\e_0}}$.

Let $\chi_0$ be a smooth convex function
vanishing solely  on $(-\infty,1]$. Let
$$
 \hat N\subset B_{\e_0}\colon \hat r(z):=r^*(z)+C^*\chi_0(\e_1^{-2}|z|^2)<0.
$$
For $C^*>0$ sufficiently large,   $\hat r$
is strictly convex on $B_{\e_0}$ and $\hat N$ is connected and
   relatively compact in $B_{\e_0}$. We remark that $\var_i$ depend
on
the first and second-order derivatives of $r$ at $p$, while $\e_0$ and $\e_1$ depend
on the least upper bound of $|\pd_zr(p)|^{-1}$ as well as on the norms of the
first and second order derivatives of $r$.
Define $\psi=\var_2\var_1\var_0$ and $N=\psi^{-1}(\hat N)$.

 Let $\chi_1$ be a smooth function
on $\rr$  that  is $1$ for $|t|<1$ and $0$ for $|t|>2$. Define
\ga\nonumber
\tilde D\colon \tilde r(z):=r(z)-\del\chi_1(\e_2^{-2}|z-p|^2)<0; \quad B=N\cup(\tilde D\setminus D).
\end{gather}
Here  $0<\e_2<\e_1/{C_*}$ for some $C_*$ that depends only on the least upper bound
of $|\pd_zr(p)|^{-1}$ and the norms of the first two derivatives of $r$;
and  $\psi(B_{\e_2}(p))$ is contained in $B_{\e_1}(0)$.
 When $\del>0$ is sufficiently small, $\tilde r$  is still strictly plurisubharmonic near $\pd D$.
 Note that  the bump $B$ covers a relatively large portion of boundary
of $D$ as
$$
 p\in \pd D\cap B_{\e_2}(p)\subset B\cap\pd D.
$$
The $\e_0,\e_1,\e_2$, and $\del$ can be chosen uniformly
when $p$  varies on  $\pd D$.

Using the {\it same}
 defining function $r$, we find finitely many points $p_1,\ldots, p_m$ in $\pd D$, the associated biholomorphic mappings $\var_{2,p_i}, \var_{1,p_i},\var_{0,p_i}$, depending on $p_i$,
    and the biholomorphisms
  $\psi_i=\var_{0,p_i}\var_{1,p_i}\var_{2,p_i}$
defined on  $\om_i=\cL U(p_i)$ such that $\psi_i(\om_i)=B_{\e_0}$, $\psi_i(p_i)=0$,
 $\psi_i(B_{\e_2}(p_i))\subset B_{\e_1}$, while $\{B_{\e_2}(p_i)\}$ is an open
 covering of $\pd D$. Set $r_0=r$ and
 \gan
  \hat N_i\subset B_{\e_0}\colon \hat r_i(z):=r_{i}^* (z)+C^*\chi_0(\e_1^{-2}|z|^2)< 0,\\
  D_{i+1}\subset\cL U(\ov D)\colon  r_{i+1}(z):=r_{i}(z)-\del^*\chi_1(\e_2^{-2}|z-p_i|^2)<0
 \end{gather*}
 with $r_i^*=e^{r_{2,p_i}\circ\var_{2,p_i}^{-1}}-1, r_{2,p_i}=r_{1,p_i}\circ\var_{1,p_i}^{-1}$, and $  r_{1,p_i}=\f{1}{2|r_z(p_i)|}r_i\circ\var_{0,p_i}^{-1}.
 $
 Set $N_i=\psi_i^{-1}(\hat N_i)$ and
 $B_i=N_i\cup( D_{i+1}\setminus D_{i})$.  When $\del^*>0$ is sufficiently small,  $r-r_i$
 have small $C^2$ norms. Thus, we may assume that the $\e_i,\delta$ have been so chosen
 that the $\hat r_i$,
$\hat N_i$ are strictly convex, $r_i$ are strictly plurisubharmonic near $\pd D_i$, and \re{NiDi} holds.
Note that $r_i$ are defined on the domain of $r$   and $r_{i+1}\leq r_i$.
Since $\{B_{\e_2}(p_i)\}$ covers $\pd D$ and $B_{\e_2}(p_i)\cap \pd D\subset D_{i+1}$, then
$\ov D\subset D_m$ as claimed.

We now consider the family $\{\ov {D^t}\}$.  Fix $t_0$. We apply the above construction
to the domain $D=D^{t_0}$.  We rename the above $D_i, N_i, B_i, \hat r_i, r_i$ by $D_i^{t_0}, N_i^{t_0}, B_i^{t_0}, \hat r_i^{t_0}, r_i^{t_0}$, respectively, while the $\psi_i$ is a biholomorphic mapping  from $\om_i$ onto $B_{\e_0}$.
 By \rp{nonparem}, we find
a family $\{\Gaa^t\}$ of embeddings from $\ov D$ onto $\ov{D^t}$, where $t$ is defined on
$\ov I$ and $I$ is a neighborhood of $t_0$ in $[0,1]$.  By the parametrized version of
Seeley extension (\rl{seeleyp}),  we may assume that  $\{\Gaa^t\}\in C^{a+2,j}(\ov U\times\ov I)$
with  $\Gaa^t$ being extended  to  embeddings from  $ \ov U$ onto $ \ov{U^t}$. Here $\ov D\subset U$.  Replacing $\Gaa^t$ by $\Gaa^t\circ(\Gaa^{t_0})^{-1}$, we may assume that $\Gaa^{t_0}$ is
the identity on $U$.
We may also assume that $\ov{D_{m+1}^{t_0}}$ is contained in $U^{t_0}$.
Fix $t_0$ and define
\gan
D_i^t=\Gaa^t (D_i), \quad
B_i^t=\Gaa^t (B_i^{t_0}), \quad
N_i^t=\Gaa^t (N_i^{t_0}), \quad
\hat N_i^t=\psi_i(N_i^t),\\
\hat r_i^t=\hat r_i^{t_0}\circ\psi_i\circ(\Gaa^t)^{-1}\circ(\psi_i)^{-1}, \quad
 r_i^t= r_i^{t_0}\circ(\Gaa^t)^{-1}.
\end{gather*}
This gives us $(ii)$.  We obtain $(iii)$ as follows.
When $I$ is sufficiently small, $\hat r_i^t$ is a strictly convex defining function of $\hat N_i^t$
on $\ov{B_{\e_0}}$ and
$r_i^t$ is a  defining function of $D_i^t$ on $U$ that is strictly plurisubharmonic near
$\ov U\setminus D^t$. Also, $\{\hat r_i^t\}\in
C^{a+2,j}(\ov{B_{\e_0}}\times\ov I)$ and $\{r_i^t\}\in C^{a+2,j}(\ov U\times\ov I)$.
 \end{proof}

  \setcounter{thm}{0}\setcounter{equation}{0}

\section{Boundary regularity for variable strictly pseudoconvex domains}
\label{sect:breg}

In this section, we study the boundary regularity of the $\db$-equation on variable strictly pseudoconvex domains.  The solutions  are obtained first for  strictly convex domains. Using a reduction procedure via Grauert's bumps,
we then apply the regularity result  to the general domains.

Let us start with a homotopy formula constructed by Lieb-Range~\cite{LR80}.
Let $D$ be a  bounded convex domain with $C^{a+2}$ boundary with $a\geq0$.  Then $D$ has a defining function $r\in C^{a+2}(\cL U(\pd D))$ that is convex near $\pd D$.  In fact,  the signed distance function $\delta_{\pd D}$ is of class $C^{a+2}$ near $\pd D$ (see~\ci{GT01}),  and it is convex in $\cc^n$ because
$
\del_{\pd D}(x)=\sup_{D\subset H}\del_H(x),
$
where $H$ are affine half-spaces in $\cc^n$ bounded by hyperplanes.
 The convexity of $D$ implies that
$$
\RE\{r_\zeta\cdot(\zeta-z)\}>0, \quad\forall \zeta\in\pd D, z\in D.
$$
(Recall that in our convention, a convex set is connected.)  Let
$
g^0(z,\zeta )=\ov\zeta-\ov z, g^1(z,\zeta )=r_\zeta,
$
and $w=\zeta-z$.  Define  \gan
\omega^\ell=\f{1}{2\pi i}\f{g^\ell\cdot dw}{g^\ell\cdot w},
\quad
\Omega^\ell=\omega^\ell\wedge(\ov\pd\omega^\ell)^{n-1},\\
\Omega^{01}=\omega^0\wedge\omega^1\wedge\sum_{\alpha+\beta=n-2}
(\ov\pd\omega^0)^{\alpha}\wedge(\ov\pd\omega^1)^{\beta}.
\end{gather*}
Here $\Om^{01}$ is $\om^0\wedge\om^1$ when $n=2$ and it is zero for $n=1$. Note that
\eq{oidb}
\om^\ell\wedge(\db\om^\ell)^\all=
\frac{g^\ell\cdot dw\wedge(\db (g^\ell\cdot dw))^\all}{(2\pi i \, g^\ell\cdot w)^{\all+1}}.
\eeq
Decompose $\Om^\ell=\sum\Om_{0,q}^\ell$ and $\Om^{01}=\sum\Om_{0,q}^{01}$,
where $\Om_{0,q}^\ell,\Om_{0,q}^{01}$
are of $(0,q)$-type  in $z$. We have
\gan
\db_\zeta\Om_{0,q}^0+\db_z\Om_{0,q-1}^0=0,   \quad q\geq 1,\quad
\db_\zeta\Om^{01}_{0,q}+\db_z\Om_{0,q-1}^{01}=\Om_{0,q}^0-\Om_{0,q}^1.
\end{gather*}
We get the homotopy formula for $(0,q)$ form $f$:
\ga\label{chtf--}
f(z)=\db_zT_qf+T_{q+1}\db_zf, \quad z\in D, \quad 1\leq q\leq n,\\
T_qf=-\int_{\pd D}\Om_{0,q-1}^{01}\wedge f+\int_D\Om_{0,q-1}^0\wedge f,
\label{chtf-}
\quad q\geq1,\\
f(z)=\int_{\pd D}f(\zeta)\Om_{0,0}^1(z,\zeta ), \quad\forall z\in D, f\in C^1(\ov D)\cap\mathcal A(D).
\label{lerayf-}
\end{gather}
(See~\cite[p.~273]{CS01}.)

 The formulas  \re{chtf--}-\re{lerayf-} are valid, provided that $r_\zeta$ can be replaced by a
{\it Leray map}   $g^1(z,\zeta )$, i.e. it is
 holomorphic in $z\in D$ and
\eq{leray}
g^1(z,\zeta )\cdot(\zeta-z)\neq0, \quad\forall\zeta\in\pd D,  z\in D.
\eeq
  Note that $r_\zeta$ is never a Leray map when $D$ is not connected.  A Leray map $g^1$ always exists when $D$
  has strictly pseudoconvex $C^2$ boundary.  In the latter case there is another homotopy formula constructed via \re{leray}, where $T_qf$, restricted to a component $\tilde D$ of $D$, is defined by \re{chtf-} in which $D$ is replaced by $\tilde D$. Furthermore, for such a homotopy formula, one only needs a mapping $g^1$ satisfying
$$
g^1(z,\zeta )\cdot(\zeta-z)\neq0, \quad\forall\zeta\in\pd \tilde D,  z\in \tilde D
$$
for each component $\tilde D$ of $D$.
\begin{rem}\label{2leray} With a Leray mapping satisfying \re{leray}, we   have the Leray formula
$$
f(z)=\int_{\pd D}f(\zeta)\Om_{0,0}^1(z,\zeta )=\int_{\pd\tilde D}f(\zeta)\Om_{0,0}^1(z,\zeta ), \quad\forall z\in\tilde D,
$$
for $ f\in C^1(\ov D)\cap\mathcal O(D).$ However, the first integral representation is more convenient in holomorphic approximation.
\end{rem}
 Note
that the  classical solution operator $T_q$ can be estimated for $\db$-closed
$(0,1)$ forms; see Siu~\cite{Si74}. For $(0,q)$ forms
 we recall a $\db$-solution operator  $S_qf$  due to Lieb-Range~\cite{LR80}.
 We
reformulate the Lieb-Range solution operator in terms of the Leray-Koppelman forms
  for a  convex domain.

  Recall that the Seeley extension $Ef$
  for a differential form $f$ on $\ov D$ is obtained by applying $E$
  to  the coefficients of the form.
  \pr{}  Let   $1\leq q\leq n$.
Let $D\subset\cc^n$ be a bounded convex domain with $C^{a+2}$ boundary.  Let $r\in C^{a+2}(\cL U(\pd D))$
be a defining function of $D$. 
Suppose that $f\in C^{1}_{(0,q)}(\ov D)$ is
 $\db$-closed  on
$\ov D$. Let $Ef$ be a $C^1$ Seeley extension of $f$ that has compact
support
in $\cL U(\ov D)$.  Then $\db S_qf=f$ on $D$ for
\ga\label{tsqf}
S_qf=L_qEf+K_q\db Ef,\\
L_qEf= \int_{\cL U(\ov D) }\Om^0_{0,q-1}\wedge Ef,
\quad K_q\db Ef=
 \int_{\cL U(\ov D)\setminus D}\Om_{0,q-1}^{01}\wedge \db_\zeta Ef.
 \label{leqe}
\end{gather}
\epr
\begin{proof}  Let us modify the solution operator $T_q$   given by \re{chtf--}-\re{chtf-}.
The $\Om^{01}$ has total degree $2n-2$.
Since $Ef$ has compact support in $\cL U(\ov D)$, we apply Stokes' formula and get
\aln
&-\int_{\pd D}\Om_{0,q-1}^{01}\wedge f=\int_{\cL U(\ov D)\setminus D}\db_\zeta\Om_{0,q-1}^{01}\wedge f
+\int_{\cL U(\ov D)\setminus D}\Om_{0,q-1}^{01}\wedge\db Ef\\
&\ =-\int_{\cL U(\ov D)\setminus D}\left(\db_z\Om_{0,q-2}^{01}\wedge f-\Om_{0,q-1}^0\wedge f+\Om_{0,q-1}^1\wedge f\right)
+\int_{\cL U(\ov D)\setminus D}\Om_{0,q-1}^{01}\wedge\db Ef\\
&\
=-\db_z\int_{\cL U(\ov D)\setminus D}\Om_{0,q-2}^{01}\wedge f+\int_{\cL U(\ov D)\setminus D}
\left(
\Om_{0,q-1}^0\wedge f-\Om_{0,q-1}^1\wedge f
+ 
\Om_{0,q-1}^{01}\wedge\db Ef\right).
\end{align*}
Let us look at the $4$ integrals after the  last equal sign.
To modify the solution operator, we remove the first integral as $\db^2=0$. The third integral of the $4$ terms
is $0$ when $q>1$, or holomorphic when $q=1$.
In the latter case,
 we   remove it. We end up with two integrals that do not involve boundary integrals.
Moreover, the second integral,
after   combined with the last integral in \re{chtf-},
 is   over the domain $D\cup(\cL U(\ov D)\setminus D)=\cL U(\ov D)$.
We have verified \re{tsqf}.
\end{proof}


We now apply $S_qf$ to our first case where   $D^t$ are strictly convex. More precisely, we assume
that   $D^t$ are strictly convex and have defining functions  $r^t$    on $U^t$ which are strictly convex near $\ov{U^t}\setminus D^t$, while $\{r^t\}\in C^{2,0}\{\ov{U^t}\}$.
We replace the above $D,\cL U(\ov D), r, E$,  $g^1$ by  $D^t$, $U^t$,  $r^t$,  $E^t$,
$\DD{r^t}{\zeta}$ respectively.
Let $S_q^tf$ be the operator $S_q$
applied to $(0,q)$ form $f^t$ on $D^t$.  Thus, we have
\eq{stqf}
S^t_qf=L_q^tEf+K_q^t\db Ef,
\eeq
where $L_q^tEf$, $K_q^t\db Ef$ are
 defined by \re{leqe} in which $Ef$, $\cL U(\ov D)$ are  replaced by $E^tf$, $\ov {U^t}$ respectively.

In real coordinates $x,\xi$, we will write $z_j=x_j+ix_{n+j}$ and
$\zeta_j=\xi_j+i\xi_{n+j}$. Recall that $\hat\pd_\xi^ir^t$
  denotes the set of derivatives of $r^t(\zeta)$ of order at most $i$.
   In view of \re{oidb},  we can express the coefficients of $S_q^tf$ as follows.
\pr{klformula} Let $L_q^t, K_q^t$ be given by \rea{leqe} and \rea{stqf}.
 The coefficients of the $(0,q-1)$ form
 $K_q^t\db Ef(z)$ are
  $\cc$-linear combinations of
\eq{ktfz}
 K^tf_1(z):=\int_{\xi\in U^t\setminus D^t}f_1^t(\xi)\f{A(\hat\pd^2_{\xi}r^t,\xi,x)(\xi_i-x_i)}{(r^t_{\zeta}\cdot(z-\zeta))^{n-m}|\zeta-z|^{2m}}\, dV(\xi), \quad\forall x\in D^t
 \eeq
with $m=1,\ldots, n-1$ and $i=1,\dots, 2n$, where $f_1^t$  is a    coefficient of $\dbar E^tf$,
 $A$ is a polynomial,   and $dV$ is the standard volume-form on $\cc^n$.
The coefficients of $(0,q-1)$ form $L_q^tf(z)$ are $\cc$-linear combinations of
 \eq{ltfz}
L^tf_0(z):=\int_{\xi\in U^t}  f_0^t(\xi)\f{\xi_i-x_i}{|\zeta-z|^{2n}}\, dV(\xi), \quad 1\leq i\leq 2n,
\eeq
where   $f_0^t$ are coefficients of $E^tf$. The coefficients for the $\cc$-linear combinations  are universal and independent of $t$.
 \epr
By the proposition, it suffices to show that $\{K^tf_1\}$, $\{L^tf_0\}$ are in $C^{k+1/2,j}(\ov{\cL D})$,
when $\{f_1^t\},\{f_0^t\}$ are in $C^{k,j}(\{\ov{U^t}\})$.

We first state the following interior  estimate.
It is valid for $C_*^{k+\all, j}$ norm for  integers $k,j$. See~\cite{We89} for  fixed domains.
\pr{estL} Let $j,k\in\nn$ and  $0<\all<1$. Let $\{{U^t}\}$
  be a smooth family of domains
   of $C_*^{k+1+\all,j}$ boundary. Let $\{f^t\}\in C_*^{k+\all,j}(\{\ov{U^t}\})$ and
let $\{L^tf\}$ be defined by \rea{ltfz}.  If $\cL D$ is a relatively compact open subset of the total space
$\cL U$ of $\{ U^t \}$, then for  $s_{k}\leq Ck$,
$$
|Lf|_{\cL D;k+1+\all,j}\leq C\dist(\cL D,\pd \cL U)^{-s_k} |f|_{\cL U;k+\all,j}.
$$

\epr
\begin{proof}  This follows directly from the classical estimates for the Newtonian potential.    Since $\cL D$
is relatively compact in $\cL U$, we find a smooth function $\chi^t(x)$ in $t,x$ such that $\chi^tf^t$ has compact support in $\cL U$, while $\chi^t(x)$ equals $1$ near $\ov{\cL D}$.  We may assume that $f^t=\chi^tf^t$ and that for   $\{f^t\} \in C_*^{k+\all,j}$ we have
$$
\pd_t^j\{L^tf(z)\}=\int_{B_R} \pd_t^j\{ f^t(\xi)\}\f{\xi_i-x_i}{|\zeta-z|^{2n}}
\, dV(\xi), \quad z\in\ov{ D^t}\subset B_R.
$$
The estimate in $x$ derivatives is then classical.
For detail, see~\ci{GT01} (pp.~54-59).
\end{proof}
When $n=1$, we have $S_1f=L_1Ef$ in \re{tsqf}. Thus we have proved the following one-dimensional result.
\th{spc1} Let  $j,k\in\ov\nn$ with  $k\geq j$.
Let    $\{D^t\}$ be  a smooth family of non-empty bounded domains  in $\cc$ with $C^{k+1+\all,j}$ boundary.
Assume that   $f^t$ are   $(0,1)$ forms on $D^t$ with  $f=\{f^t \}
\in C^{k+\all,j}(\ov{\cL D} )$.
Then exists  linear solution operators $u^t=S^tf$ to  $\db u^t=f^t$ on $D^t$   so that
$\{S^tf\}\in C^{k+1+\all,j}(\ov{\cL D} )$. \eth

We now estimate the boundary integral to gain $\f{1}{2}$ in the H\"older exponent.
\pr{estK} 
Let  $j,k\in\nn$ with $j<k$.  Let   $\{U^t\}$
  be a smooth family of bounded domains with $C^{k+1,j}$  boundary.  Suppose that each $D^t$ is non empty
and  relatively compact
in $ U^t$ and it is defined by $r^t<0$ on $U^t$ with $dr^t\neq0$ on $\pd D^t$. Suppose that the real Hessian of $r^t$ is strictly positive-definite on $\ov {U^t}\setminus D^t$ and $\{r^t\}$ is of class $C^{k+1,j}(\{\ov {U^t}\})$.
 Let $\{K^tg\}$ be defined by \rea{ktfz}.
If $g \in C^{k-1,j}(\{\ov {U^t}\})$, $g^t\equiv0$ on $D^t$,  then
$$
\|Kg  \|_{\cL D; k+1/2,j}\leq C \|g \|_{\cL U;k-1,j}.
$$
where $\cL D$, $\cL U$ are respectively the total spaces of $\{D^t\}, \{U^t\}$.
\epr
\begin{proof}  By \rp{nonparem} we know that $\{D^t\}$ is a family of bounded domains of   $C^{k+1,j}$
 boundary. Using a cut-off function we may assume that each $g^t$ has a compact support in a small neighborhood of $\pd D^t$.      Thus the integral $K^tg$ is over a fixed bounded domain, which simplifies the computation of   $t$ derivatives.

 Note that $r^t_{\zeta}\cdot(\zeta-z)\neq0$,  for $z\in D^t$ and $\zeta\in U^t\setminus D^t$.
 The latter contains
the support
of $g^t$.  First, we will take the derivatives on the integrand directly.
We denote by $N_\nu(x)$ a monomial of degree $\nu$ in $x$.  Let $A(w)$ denote a polynomial in $w$.  Also,
the $A$ might be different when
it recurs.
 Let us write
\gan
K^tg(z):=\int_{\cc^n} g^t(\xi )\f{A(    \hat\pd_{\xi }^2r^t,\xi,x)N_{1}(\xi -x )}
 {(r^t_{\zeta }\cdot(\zeta -z ))^{n-\ell}|\zeta -z |^{2\ell}}\, dV(\xi),\quad   z\in D^t.\end{gather*}
Note that $\{g^t\}$ is only in $C^{k-1,j}$.
We  first compute $\pd_t^i\pd_x^{k_1}\{K^tg(z )\}$ for $i\leq j$ and $i+k_1<k$. We  then apply   the integration by parts to derive a new formula.
Finally we  compute two more derivatives to derive the $\f{1}{2}$-estimate by using the Hardy-Littlewood lemma.

We   write
$
\pd_{x}^{k_1}\{K^tg(z )\}$ as a linear combination of  $  K_1^tg(x)$ with
\ga\nonumber
K_1^tg(x):=\int_{\cc^n} g^t(\xi )\f{A(  \hat\pd^2_{\xi }r^t,\xi,x)
N_{1-\mu_0+\mu_2}(\xi -x )}
 {(r^t_{\zeta }\cdot(\zeta -z ))^{n-\ell+\mu_1}|\zeta -z |^{2\ell+2\mu_2}}\, dV(\xi),\\
 \label{aaa}
\mu_0+\mu_1+\mu_2\leq k_1, \quad 1-\mu_0+\mu_2\geq0.
 \end{gather}
It suffices to estimate
$
\pd_t^i  \{K_1^tg(x)\}$, which
is a linear combination of
\ga\label{Jzt}
 J^t(x):=\int \f{ \hat g^t(\xi )  N_{1-\mu_0+\mu_2+j_2}( \xi-x)}
 {(r^t_{\zeta}\cdot(\zeta-z))^{n-\ell+\mu_1+j_2}|\zeta-z|^{2\ell+2\mu_2}}\,  dV(\xi)
 \end{gather}
 with \eq{hatg}
 \hat g^t(\xi )
 =  \pd_t^{j_0}(g^t(\xi ))
 A(  \pd_t^{j_1} \hat \pd^2_{\xi}r^t,\pd_tr_\zeta^t,   \xi,x).
\eeq
 Furthermore,
    $\mu_0,\mu_1,\mu_2$ satisfying \re{aaa} and
 $$
 j_0+j_1+j_2\leq i, \quad i\leq j,\quad
 i+k_1\leq k-1.
 $$

   Next, we will apply the integration by parts to  reduce the exponent
    of $r_{\zeta}^t\cdot(\zeta-z)$  to  $n-\ell$.
   This requires us to transport the derivative in $t$ to derivatives in $\xi$.
Thus,    we need  the space $C^{\bullet}$ instead of $C^{\bullet}_*$.
     To this end we write \re{Jzt}   as
 \ga
  J^t_{m}(x)=\int_{\cc^n}\f{ \tilde g^t(x,\xi)}{(r^t_{\zeta}\cdot(\zeta-z))^{n-\ell+m}}\, dV(\xi),\quad m=\mu_1+j_2,
\label{mm1}\\
  \tilde g^t(x,\xi)= \hat g^t(\xi )\f{  N_{1-\mu_0+\mu_2+j_2}(\xi-x)}
 { |\zeta-z|^{2\ell+2\mu_2}}.
\label{tildeg}
 \end{gather}
 Using a
 partition of  unity in $(\zeta, t)$ space, we may assume that $  g^t(\zeta)$ has compact support in a small ball $B$
 centered at $(\zeta_0,t_0)$, and on $B$
 $$
 \pd_{\zeta}\left\{r^t_{\zeta}\cdot(\zeta-z)\right\}\neq0, \quad u^t(x,\xi):=\pd_{\xi_\beta}
\left\{r^t_{\zeta}\cdot(\zeta-z)\right\}\neq0,
  $$
  for some $\beta$.  Since $ g^t(\zeta)$ is $0$ on $D^t$ and has a compact support in $B$,
  we apply Stokes' theorem and obtain that,  up to a constant multiple
  $$
  J^t_{m}(x)=\int_{\cc^n} \f{\pd_{\xi_\beta}\{u^t(x,\xi)^{-1}\tilde g^t(x,\xi)\}}{(r^t_{\zeta}\cdot(\zeta-z))^{n-\ell+m-1}}\, dV(\xi),
  \quad\forall z\in D^t. $$
 Repeating this shows that up to a constant
 $$
 J_{m}^t(x)=\int_{\cc^n} \f{v^t(x,\xi)}{(r^t_{\zeta}\cdot(\zeta-z))^{n-\ell}}\, dV(\xi),
  \quad\forall z\in D^t$$
 with
 \eq{vtxx}
 v^t(x,\xi):=
 (\pd_{\xi_\beta}\circ u^t(x,\xi)^{-1})^{m}\{ \tilde g^t(x,\xi)\}. 
   \eeq

Since $\tilde g^t(x,\xi)=0$ for $\xi\in D^t$, it is easy to see that $J^t_m(x)$, $\pd_xJ^t_m$ are continuous on $\mathcal D$.  To show that $\{K^tg\}$ is in $C^{k+1/2,j}(\{\ov{ D^t}\})$,  by Hardy-Littlewood lemma it suffices to verify
 $$
 |\pd_{x}^2J^t_m(x)|\leq C\dist(x,\pd D^t)^{-1/2}, \quad \forall x\in D^t.
 $$
 Since $g^t(x)= 0$ for $x\in D^t$, we  obtain that  for $ b\leq j, a+b\leq k-1$
\aln
|\pd_{\xi_\beta}^{a}\pd_t^{b}\{ g^t(\xi )\}|&=|\pd_{\xi_\beta}^{a}\pd_t^{b}\{ g^t(\xi )\}-\pd_{x_\beta}^{a}\pd_t^{b}\{ g^t(x)\}|\leq C\|g\|_{\cL U; k-1,j}|\xi-x|^{k-1-a-b}.
\end{align*}
Thus, $|g^t(\xi )|\leq C|\xi-x|^{k-1}$, $|\hat g^t(\zeta)|\leq C|\xi-x|^{k-1-j_0}$ by \re{hatg}, and
$|\tilde g(x,\xi)|\leq C|\xi-x|^{\nu}$ by \re{tildeg}, for
$$
\nu=(k-1-j_0)+(1-\mu_0+\mu_2+j_2)-(2\ell+2\mu_2).
$$
By \re{vtxx}, we get
\ga\nonumber
|v^t(x,\xi)|\leq C(\|r\|_{\cL U;k+1,j})\|g\|_{\cL D; k-1,j}|\xi-x|^{\nu'}, \quad \nu'=\nu-m.\end{gather}
Recall from \re{mm1} that $m=\mu_1+j_2$. Hence\aln
\nu'&=
 (k-1-j_0)+(1-\mu_0+\mu_2+j_2)-(2\ell+2\mu_2 )-(\mu_1+j_2)\\
&= k-j_0-\mu_0-2\ell-\mu_2-\mu_1\geq k-i-k_1-2\ell\geq 1-2\ell.\end{align*}
 We    obtain similar estimates for $|\pd_{x}^sv^t(x,\xi)|$. In summary, we have
\gan
|\pd_{x}^s v^t(x,\xi)|\leq C\|g\|_{\cL D; k-1,j}|\xi-x|^{(1-2\ell)-s}, \quad s=0,1,2.
\end{gather*}
Here $C$ depends on $\|r\|_{\cL U;k+1,j}$.
 Now, $\pd_{x}^2J_m(x)$ is a linear combination of
 \begin{gather*}
 J^t_{m,0}(x):= \int_{\cc^n}
 \f{\pd_{x}^2\left\{v^t(x,\xi)\right\}}{(r^t_{\zeta}\cdot(\zeta-z))^{n-\ell}}\, dV(\xi),
 \\
   J^t_{m,1}(x):= \int_{\cc^n}
 \f{A(r^t_{\zeta},\xi,x)\pd_{x}\left\{v^t(x,\xi)\right\}}{(r^t_{\zeta}\cdot(\zeta-z))^{n-\ell+1}}\, dV(\xi),
 \\
  J^t_{m,2}(x):= \int_{\cc^n}
 \f{A(r^t_{\zeta},\xi,x)v^t(x,\xi)}{(r^t_{\zeta}\cdot(\zeta-z))^{n-\ell+2}}\, dV(\xi),
  \end{gather*}
for $ x\in D^t.$
Therefore, we obtain
\ga\label{Jmiz}
 |J^t_{m,i}(x)|\leq C\int_{U^t\setminus D^t}\f{d V(\xi )}{|r^t_{\zeta}\cdot(\zeta-z)|^{n-\ell+i}|\zeta-z|^{2\ell +1-i}}.
 \end{gather}
Since
$|r_{\zeta}^t\cdot(\zeta-z)|\geq C|\zeta-z|^2$,   it suffices to estimate  \re{Jmiz} for $\ell=n-1$.  We have
$$
|J^t_{m,i}(x)|\leq 
C\int_{U^t\setminus D^t}
 \f{1}{|r^t_{\zeta}\cdot(\zeta-z)|^{i+1}|\zeta-z|^{2n-i-1}}\, dV(\xi ), \quad  i=0,1,2.
 $$
Then the last integrals are bounded by $C\dist(z,\pd D^t)^{-1/2}$. For the further detail, see Lieb-Range
  (the estimates of $J_k(z)$ in~\cite{LR80},  pp. 155--166.)
\end{proof}

To study regularity of $\db$-solutions for variable domains, we need to introduce the following.
\defn{A
 family $\{E^t\}$ of subsets in a topological space $X$ is
  {\it 
  upper semi-continuous}, if
for each $t$ and every open neighborhood $U$ of $E^t$ in $X$, we have
$
E^s\subset U
$
when $|s-t|$ is sufficiently small.
}

It is easy to see that the family $\{X\setminus E^t\}$ might not be upper semi-continuous in $X$ when $\{E^t\}$ is  upper semi-continuous in $X$. \begin{rem}The family    of boundaries of domains defined by $r^t<c$
is not necessarily upper
semi-continuous.  \end{rem}
\le{upss}
A family   $\{K^t\}$ of compact sets $K^t$ in $\rr^d$ is upper semi-continuous
if and only if its total space $\cL K$ is compact. In particular, if  $\{K^t\}$ has a compact total space $\cL K$ and
the total space of $\{\om^t\}$ is an open subset of $\cL K$, then  $\{K^t\setminus\om^t\}$ is upper semi-continuous.
\end{lemma}
\begin{proof}
Obviously,
the last assertion follows from the first assertion. Suppose that $\{K^t\}$ is upper-semi continuous with total space $\cL K$. If $\cL K$ is not compact,  there is a sequence $(x_m,t_m)\in\cL K$ that does not admit any convergent subsequence with limit in $\cL K$. We may assume that $t_m\to t_0$ as $m\to\infty$. Since $K_{t_0}$ is bounded,   it is contained in an open ball $U$ of finite radius. By the upper-semi continuity, we know that $K_{t_m}\subset U$ for $m$ sufficiently large. We may assume that $x_m\to x_0$.
Then $x_0$ is not in $K_{t_0}$.  Take another open set $U'$ containing $K_{t_0}$ such that $x_0$ is not in $\ov{U'}$.
By the upper semi-continuity, we have $K_{t_m}\subset U'$. Then $x_0$ is in $\ov{U'}$, a contradiction.

  Suppose now that $\cL K$ is compact. Fix $t_0\in[0,1]$. Let $U$ be an open neighborhood of $K^{t_0}$ in $\rr^d$. Suppose that $K^{t_m}$ is not contained in $U$  for a sequence  $t_m\to t_0$. Take  $x_m\in K^{t_m}\setminus U$. Since $\cL K$
is compact, taking a subsequence if necessary we conclude that $(x_m,t_m)$ tends to $(x_0,t_0)\in\cL K$. This shows that $x_0\in K^{t_0}\setminus U$ and  the latter is non-empty, a contradiction.
\end{proof}

\th{spc} Let  $j,k\in\ov\nn$ with  $k-1\geq j$.
Let    $\{D^t\}$ be  a smooth family of non-empty bounded domains  in $\cc^n$ with $C^{k+1,j}$ boundary.  Assume that      $D^t$ are strongly pseudoconvex.
Assume that   $f^t$ are $\db$-closed $(0,q)$ forms on $D^t$ with  $q>0$ and $f =\{f^t \}
\in C^{k,j}(\ov{\cL D} )$.
There exist  linear solution operators $u^t=S^tf$ to  $\db u^t=f^t$ on $D^t$   so that
$\{S^tf\}\in C^{k+1/2,j}(\ov{\cL D} )\cap C^{k+\e,j}(\cL D)$ for all $\e<1$. \eth
\begin{proof}
We first consider the case when all   $\ov{D^t}$ are strictly convex.  By \rp{nonparem}
we can find defining functions $r^t$ for $D^t$, where $\{r^t\}\in C^{k+1,j}(\ov{\cL U})$;
and $r^t$ have positive-definite   real Hessian on $\pd D^t$ replacing $r^t$ by $e^{Cr^t}-1$ if necessary.
  Then we have homotopy formula \re{stqf} that provides
a solutions operators $S^t$. The regularity follows from Propositions~\ref{estL} and \ref{estK}.

The proof for the general case consists of two steps.
We   first use the bumps in \rl{naragrau}  and the theorem for the strictly convex
domains to  extend $\{f^t\}$ to a family of $\db$-closed forms on  larger domains.
 We then solve the $\db$-equation on
a fixed large domain by using the classical homotopy formula.
Note that we only constructed bumps uniformly in $t$ for $t$ close to a given value. Thus, we will first
define the solution operator $S^t$ locally in $t$ and we will then define $S^t$ for all $t$ by using a partition of unity in parameter $t$.

We recall the construction  from \rl{naragrau}. Fix $t_0$. We can find a connected neighborhood $I$
of $t_0$ such that when restricting $t$ to $\ov I$ we have the following:
there are  finitely many
strictly pseudoconvex domains $D_i^t, B_i^t, N_i^t$  with $C^{k+2}$ boundary  such that
\ga\nonumber
D_{i+1}^t= D_i^t\cup B_{i}^t,\quad \ov {D^t}\subset D_m^t,\quad
N_{i}^t=D_i^t\cap B_{i}^t,\quad D_0^t=D^t;\\
\label{sepa}
  (\ov {B_{i}^t\setminus D_i^t})\cap (\ov{D_{i}^t\setminus N_i^t})=\emptyset,\\
  \Gaa^t(D_i)=D_i^t, \quad \Gaa^t(N_i)=N_i^t, \quad \Gaa_i^t(B_i)=B_i^t
  \label{sepa+}
  \end{gather}
and there exists a biholomorphic mapping $\phi_i$ from $\om_i$ onto $B_\e$, independent of $t$,
 such that
$\hat N_i^t:=\psi_i({N_i^t})$  is strictly convex.  Furthermore, $B_i, N_i$ and $D_i$ are relatively compact
in $U$ and $\{\Gaa^t\}\in C^{k+1,j}(\ov U)$, and
the $D_i^t$ {(}resp. $\hat N_i^t)$ is defined by $r_i^t<0$ {(}resp. $\hat r_i^t<0${)}.
The  $\{r_i^t\}$  is  in $ C^{k+1,j}(\ov{U})$ and $\hat r_i^t$ is strictly convex on $\ov{ B_{\e}}$.  Here $U$ contains $\ov {D^t}$ and has $C^{k+1}$ boundary.

 Let $S_{\hat N_i^t}$ be the Lieb-Range solution operator determined by $\hat r_i^t$
 for $\hat N_i^t$. We pull back the solution operator
 to  $N_i^t$ and define $S^t_{ N_i}g:=(\psi_i)^*S_{\hat N^t_i}(\psi_i^{-1})^*g^t$.

By \re{sepa+} we know that
$$
  (\ov{B_i^t\setminus    D_i^t})\cap\pd D_i^t=\Gaa^t((\ov{B_i\setminus   D_i})\cap\pd D_i),
\quad
\ov{D_{i}^t\setminus N_i^t}=\Gaa^t(\ov{D_{i}\setminus N_i})
 $$
 are upper semi-continuous in $\cc^n$.
Thus we can   find open  neighborhoods $U_i^0, U_i^1$
 of $(\ov{B_i^{t}\setminus  D_i^{t}})$ such that $\ov{U_i^1}\subset U_i^0$ and $\ov{U_i^0}\cap(\ov{D_i^{t}\setminus   N_i^{t}})=\emptyset$ for  $t$ near $t_0$.
 Now we   find   a smooth function $\chi_i$  that
 has compact support in $U_i^0$ such that $\chi_i=1$ near $U_i^1$. Since we have only finitely many families $\{N^t_i\}$, the above construction of $S^t_{N_i}$ is valid for all $t$ near $t_0$ and all $i$.
 We then define
 \gan
 f_1^t=f^t-\db(\chi_0S^t_{N_0}f)=(1-\chi_0)f^t-(\db\chi_0)\wedge S^t_{N_0}f.
  \end{gather*}
The last identity implies that  $f_1^t$ vanishes near $\ov{(B^t_0\setminus D_0^t)}$. We extend $f_1^t$ to be zero on $\ov{D_1^t}\setminus D_0^t$. Then $f_1^t$ is $\db$-closed on $D_1^t$, and
$\{f_1^t\}\in C^{k,j}(\{\ov {D_1^t}\}).$
We define $f_{i+1}^t=f_{i}^t-\db(\chi_iS^t_{N_i}f_i)$ and extend it to   zero on $\ov{D_{i+1}^t}\setminus D_i^t$. We can write
  $$
  f_m^t=f^t-\db g^t,\quad g^t=\sum_{i=0}^{m-1}\chi_i S^t_{N_i}f_i.
  $$
We have $\{g^t\}\in C^{k+1/2,j}(\ov{\cL D})$, while
 $\{f_m^t\}\in C^{k,j}(\{\ov{D_{m+1}^t}\})$.   Since $\chi_i$ is contained in $ N_i^t$, then
 $\{\chi_iS_{N_i}^tf\}\in C^{k+\e,j}(\cL D)$ for all $\e<1$. This shows that $\{g^t\}$
 is in $C^{k+1/2,j}(\ov{\cL D})\cap C^{k+\e,j}(\cL D)$.

 Again, for the fixed $t_0$, we can find a strictly pseudoconvex domain $D_*$
of $C^2$ boundary such that $\ov{D^t}\subset D_*$ for $t$ near $t_0$.   Let $T$ be the solution operator from the classical homotopy formula   on $D_*$. By the interior regularity property of $T_{D^*}$, we get $\{T_{D_*}f_m^t\}\in C^{k+\e}(\ov D_\Gaa)$.
Then $S^tf:=T_{D_*}f_m^t+g^t$ is a solution operator of the desired property for $t$ near $t_0$. Using a partition of
unity $\{\tilde\chi_i\}$ on $[0,1]$ with $\supp\tilde\chi_i\subset(a_i,b_i)$, we obtain a solution operator $\sum\tilde\chi_i(t) S_{i}^t$, where
$S_{i}^t$ is defined for $t\in(a_i,b_i)$. Then $S^t:=\sum\tilde\chi_i(t)S_{i}^t$ has the desired properties.
  \end{proof}

\setcounter{thm}{0}\setcounter{equation}{0}

\section{Henkin-Ram\'irez functions for variable strictly pseudoconvex open sets}
\label{sect:hfp}


In this section we will construct a family of Henkin-Ram\'irez functions for variable strictly pseudoconvex open sets.
As an application, we will  find homotopy formulas for a smooth family  $\{D^t\}$ of strictly pseudoconvex domains.


The following  theorem is on    Henkin-Ram{\'{\i}}rez functions with parameter.
\begin{thm}\label{hrfn}
Let $a,b\in\ov\rr_+$, $j\in\ov\nn$, and $j\leq a$.
 Let
$\{\om^t\}, \{D^t\}, \{U^t\}$ respectively be  continuous families of domains  with
bounded total spaces $\om,\cL D, \cL U$.
 Suppose that $\om$ is relatively
compact in $\cL U$.
 Let $\{ r ^t\}$ be of class $C^{a+2,j}(\cL U)\ ($resp. $C^{b+2,j}_*(\cL U))$. Suppose that  $ r ^t$
are strictly plurisubharmonic on   $\om^t$.
 Let
  $
C^t:=\om^t\cap\{ r ^t=0\}.
  $
   Suppose that for
 each $t\in[0,1]$
 \ga
 \label{ctn0}
C^t \neq\emptyset,
 \quad  \pd D^t\subset C^t\subset\subset\om^t,
  \\
\label{rtl0}
  r ^t<0 \ \text{on $D^t
 $}, \quad r^t>0\  \text{on $\om^t\setminus (D^t\cup C^t)$},\quad
 r ^t>\del_0\  \text{on $U^t\setminus (D^t\cup\om^t)$}
\end{gather}
with $\del_0>0$.  For $\del>0$, set $D^t_\delta:=D^t \cup\{z\in\om^t\colon  r ^t(z)<\delta\}$,
$D_{-\del}^t:=\{z\in D^t\colon  r ^t(z)<-\del\}$,  and
$V^t_\del:=\{z\in D^t \cup\om^t\colon | r ^t(z)|<\del\}$.
 Assume that
\ga \label{lamb0}
\sum\DD{^2 r ^t}{\zeta_j\pd \ov\zeta_k}s_j\ov s_k\geq\la_0|s|^2,\quad\forall \zeta\in \ov{\om^t},\\
\label{pd2rho}
|\pd^2_\zeta r ^t-\pd_z^2 r ^t|<\f{\la_0}{C_n}, \quad \forall \zeta\in V^t_{\del_0},z\in\om^t\cup D^t, |\zeta-z|<d_0
\end{gather}
for some $\la_0$ and $d_0>0$ with $0<\la_0\leq|r|_{2,0}$.
Let   $0<\del_1<\del_0$ and
\eq{dd0e}
d=\min\Bigl\{d_0,\dist(  V^t_{\del_1}, \pd\om^t\setminus D^t)\colon t\in[0,1]\Bigr\},\quad
\e=\min\Bigl\{ \f{\la_0}{64}d^2,  \del_1\Bigr\}.
\eeq
   Then $d>0$ and there exist
   functions $\Phi^t(z,\zeta)$ and
  $$
   F^t(z,\zeta)
  =-\sum\DD{ r ^t}{\zeta_j}(z_j-\zeta_j)-\f{1}{2}\sum b^t_{jk}(\zeta)(z_j-\zeta_j)(z_k-\zeta_k)
$$
  so that  for $z\in D^t_{\e/2}$ and $\zeta\in D^t_{\del_1}\setminus D^t_{-\e}$ the following  hold:
 \bppp
 \item    The functions $\Phi^t(z,\zeta)$  are  holomorphic  in $z$,
and $\Phi^t(z,\zeta)\neq0$ for $z\neq \zeta$ and $ r ^t(z)\leq r ^t(\zeta)$.
\item If    $|\zeta-z|<\epsilon$, there exist $M^t(z,\zeta )\neq0$  such that
$\Phi^t(z,\zeta )=F^t(z,\zeta )M^t(z,\zeta )$ and
  \ga{}
  \label{elpp}
  \RE F^t(z,\zeta)\geq  r ^t(\zeta)- r ^t(z)+\f{\la_0}{4}|\zeta-z|^2,\quad \text{if $ |\zeta-z|<d,
  z,\zeta \in D^t_{\del_1}$}.
  \end{gather}
  \item
 The families  $\{\Phi^t\}$,  $\{M^t\}$ are in  $C^{a+1,j}(\{ \ov{D_{\e/2}^t}\times(\ov{D^t_{\delta_1}}\setminus D^t_{-\e})\} )$ $($resp.  $C_*^{b+1,j}(\{ \ov{D_{\e/2}^t}\times(\ov{D^t_{\delta_1}}\setminus D^t_{-\e})\} )$.
 \eppp
  \end{thm}
 \begin{rem} $(i)$ The main conclusion is about the uniform size  $\e$, given in \re{dd0e},  of the band
  $D^t_{\e/2}\times (D^t_{\del_1}\setminus D^t_{-\e})  $  on which $\Phi^t(z,\zeta )$ are defined.
  This will be crucial in proving a parametrized version of Oka-Weil approximation in
  section~\ref{sect:ireg}. $(ii)$
 $\pd D^t$ might not be smooth and $D^t$
might not be connected. Furthermore, $D^t$ could be empty when $C^t$ consists of local minimum points
of $\var^t$.
  $(iii)$
  The results are classical for non-parameter case.
 When $t$ is fixed and $V_\del^t$ is replaced by a neighborhood of $\pd D^t$, and $ r ^t$
is of class $C^2$, see
\cite[p.~78, Theorem 2.4.3; p.~81, Theorem 2.5.5]{HL84}.  For the case when $\pd D^t$
has finitely smooth boundary, see Range~\cite[Proposition 3.1, p.~284]{Ra86}.

 \end{rem}
\begin{proof}
To simplify notation, we first derive some uniform estimates without parameter. We then make necessary adjustments  for the parametrized version.

 Let us first assume that $ r ^t$ is independent of $t$. Write $D^t,V_\del^t,D^t_{\del}$ as
$D,V_\del,D_\del$ respectively.
We consider
the Levi polynomial of $ r $ at $\zeta$
$$
F_0(z,\zeta):=-\sum\DD{ r }{\zeta_j}(z_j-\zeta_j)-\f{1}{2}\sum\DD{^2 r }{\zeta_j\pd \zeta_k}(z_j-\zeta_j)(z_k-\zeta_k).
$$
Assume  that
\eq{zezin} \zeta\in V_{\del_1}, 
\quad |\zeta-z|<d.
\eeq
We want to verify that the $z$ satisfies the condition in \re{pd2rho} without $t$. Otherwise, let $\tilde z$ be the point closest
to $\zeta$ in the line
segment $\ov{\zeta z}$ such that $\tilde z$ is not in $\om\cup D$. Then $\tilde z\in\pd\om\setminus D$, which
contradicts the definition of $d$ in \re{dd0e}.  Note that this also shows that any $\tilde z\in\ov{\zeta z}$ satisfies
\re{pd2rho} in which $z$ is replaced by $\tilde z$.
 Let $f(t)= r ((1-t)\zeta+tz)$.
Then $f(1)-f(0)=f'(0)+\f{1}{2}f''(0)+\f{1}{2}(f''(s)-f''(0))$ for some $0<s<1$.  By \re{pd2rho}-\re{dd0e}, we get
\aln
2\RE F_0(z,\zeta)&\geq  r (\zeta)- r (z)+\la_0|\zeta-z|^2\\
&\quad-C_n'\max_t|\pd^2 r ((1-t)\zeta+tz)-\pd^2 r (\zeta)||\zeta-z|^2.
\end{align*}
Therefore,   if $z,\zeta $ satisfy \re{zezin}, then
\eq{2ref0}
2\RE F_0(z,\zeta)\geq  r (\zeta)- r (z)+\f{\la_0}{2}|\zeta-z|^2.
\eeq
Using a real smooth  function   $\chi\geq0$ with compact support in the unit ball of $\cc^n$ so that $\int\chi=1$,  let us define
$$
\chi_d(z)=d^{-2n}\chi(d^{-1}z), \quad
a_{ij}(z)=\int \DD{^2 r }{\zeta_i\pd\zeta_j}(z-\zeta)\chi_d(\zeta)\, dV(\zeta).
$$ Then
we get   $C^{\infty}$ functions $a_{ij}$ in  $V_{\del_1}=D_{\del_1}\setminus \ov D_{-\del_1}$ such that for $\zeta\in V_{\del_1}$,
\gan
\sup_{\zeta\in\om} \left|a_{ij}(\zeta)-\DD{^2 r }{\zeta_i\pd\zeta_j}\right|<   C_n'\f{\la_0}{C_n},\\
 |a_{ij}|_{a}\leq C''_a| r |_{a+2},\quad
  |a_{ij}|_{a+1}\leq C''_ad^{-1}| r |_{a+2}.
 \end{gather*}
We replace the Levi polynomial $F_0$ by
 $$
F(z,\zeta):=-\sum\DD{ r }{\zeta_j}(z_j-\zeta_j)-\f{1}{2}\sum a_{ij}(\zeta)(z_i-\zeta_i)( z_j-\zeta_j).
$$
Now \re{2ref0} implies that
\eq{2ref}
2\RE F(z,\zeta)\geq r (\zeta)- r (z)+ \f{\la_0}{4}|\zeta-z|^2
\eeq
if $\zeta$ and  $z$ satisfy \re{zezin}. This shows that if $\zeta, z$ satisfy
\ga\label{rozz-}
d/2<|\zeta-z|<d, \quad z,\zeta \in D_{\del_1},\\
\label{rozz}
  r (\zeta)> r (z)-\f{\la_0d^2}{32},
\end{gather}
then by \re{2ref}-\re{rozz},  we must have
\eq{rozz+}
2\RE F(z,\zeta)\geq \f{\la_0}{4}\times\f{d^2}{4}-\f{\la_0d^2}{32}=\f{\la_0d^2}{32}.
\eeq
Let $\chi$ be a $C^\infty$ function such that $\chi(\zeta)=1$ for $|\zeta|<\frac{3d}{4}$ and $\chi(\zeta)=0$ for
 $|\zeta|>\f{7d}{8}$. Define
\eq{fzztwo}
f(z,\zeta)=\begin{cases} \db_z(\chi (\zeta-z)\log F(z,\zeta))
      & \text{if $d/2<|\zeta-z|<d$}, \\
   0   & \text{otherwise}.
\end{cases}
\eeq
It is clear that the coefficients of $f$ are $C^{a+1}$ in $z,\zeta $ if
$$
z\in D_{\e}, \quad \zeta\in D_{\del_1}\setminus D_{-\e}, \quad \e=\min\left\{\del_1,\f{\la_0d^2}{64}\right\}.
$$
Indeed, the latter conditions and  \re{rtl0} imply $r(\zeta)-r(z)\geq -2\e\geq -\f{\la_2 d^2}{32}$, which gives us \re{rozz}. Thus if \re{rozz-} holds, then \re{rozz+} implies that $f$
is defined in the first case of \re{fzztwo}. So it is of class $C^{a+1}$.  If \re{rozz-} fails, then $f$ is defined
in the second case. Furthermore, near $|\zeta-z|=d/2$ or $d$,
 \re{rozz} and \re{rozz+} imply that $f$ is identically zero.
 Therefore, $f$ is of class $C^{a+1}$.

 We claim that
 $ r \geq\del_0$ on $\pd\om\setminus D$ and it  takes all values $[0,\del_0)$ in $\om$. Indeed, by \re{ctn0}, $\om$ is not contained in $D$. Then $m:=\max\{r(z)\colon, z\in\pd\om\}\geq\del_0$ by \re{rtl0}. For any $c\in[0,\del_0)$, we find
a line segment $\gaa$  with $\gaa(0)\in C(t)$, $m':=r(\gaa(1))\in[c,m)$,
and $\gaa(1)\in\om$. Let $s'$ be the largest $s$ such that $r(\gaa(s))=0$, and let $s''$ be the smallest $s\in[s_*,1]$ such that $r(\gaa(s))=m'$.  Thus $\gaa((s',s''))\subset\om\setminus\ov D$ and
$r(\gaa((s',s'')))=[0,m']$; the claim is verified.
By Sard's theorem,  the $ r $
attains  a regular value   $\e'\in[4\e/5,9\e/{10}]$   on $\om$. Recall that
$
D_{\e'}=D\cup \{z\in\om\colon r (z)<\e'\}.
$
By \re{rtl0}, we obtain
$$
\pd D_{\e'}=\om\cap \{ r ={\e'}\}
$$
which is compact and of class $C^{a+2}$.

  Let $T_{D_{\e'}}$ be a  solution operator from the classical homotopy formula
for the strictly pseudoconvex domain $D_{\e'}$.
 (Note that the construction does not require the domain to be connected,
 as shown in \rt{phtf} below for the parametrized version.)

Define
$$
u(z,\zeta)=T_{D_{\e'}}f(\cdot,\zeta)(z), \quad\forall \zeta\in D_{\del_1}\setminus D_{-\e},
z\in D_{\e'}.
$$
By the interior estimate, we obtain   $u\in C^{a+1}(D_{\e'})$ as $f\in C^{a+1}(D_{\e'})$.
 We also have
$$
\DD{^{|\beta|} u(z,\zeta)}{\xi^\beta}=T_{D_{\e'}}\left\{\DD{^{|\beta|} f(\cdot,\zeta)}{\xi^\beta}\right\} (z),
$$
which is continuous on $D_{e'}\times (D_{\del_1} \setminus D_{-\e})$ for  $|\all|+|\beta|\leq a+1$.
Therefore,  
$$
T_{D_{\e'}}f(\cdot,\zeta)(z)\in C^{a+1}(D_{\e'}\times (D_{\del_1}\setminus D_{-\e})).
$$

By \re{2ref}, we can define $\log F(z,\zeta)$ for $0<|\zeta-z|<\e$ to define
$$
\Phi(z,\zeta)=\begin{cases}
 F(z,\zeta)e^{-u(z,\zeta)}    & \text{if $|\zeta-z|<d$}, \\
e^{\chi(\zeta-z)\log F(z,\zeta)-u(z,\zeta)}    & \text{otherwise}.
\end{cases}
$$
 This shows that $\Phi\in C^{a+1}(D_{\e'}\times (D_{\del_1} \setminus D_{-\e}))$.

We now consider the parametrized version.    Let us first show that $d$, defined by \re{dd0e}, is positive.
(Note that we cannot conclude that $\pd\om^t\setminus D^t$, $\pd V_{\delta_0}^t$ are upper-semi continuous.)
Otherwise, there are sequences $z_k\in V_{\del_1}^{t_k}, z_k'\in\pd\om^{t_k}\setminus D^{t_k}$ with $t_k\to t_0$
as $k\to\infty$ such that $|z_k'-z_k|\to0$ as $k\to\infty$. By \re{ctn0}, we know that $\cL D$ is relatively compact in $\cL U$ as $\om$ is relatively compact in $\cL U$. Taking a subsequence,  we may assume that $z_k\to z_0$ and $z'_k\to z_0$ as $k\to\infty$. We have $\pd\om^t\setminus D^t\subset U^t\setminus(D^t\cup\om^t)$; hence \re{rtl0}
implies that $r(z_k')>\del_0$.
We have $z_k\in V_{\del_0}^{t_k}$; hence $|r^{t_k}(z_k)|\leq\del_1$. Letting $k\to\infty$, we obtain that $r^{t_0}(z_0)\geq\del_0$ and $r^{t_0}(z_0)\leq\del_1$, a contradiction.

We first remark that  $ r ^t$ is defined on $U^t$.
For   $D^t$ defined by $ r ^t<0$ (possibly empty),  we take
$$
F^t(z ,\zeta ):=-\sum\DD{ r ^t}{\zeta _j}(z _j-\zeta _j)-\f{1}{2}\sum b^t_{ij}(\zeta )(z _i-\zeta _i)( z _j-\zeta _j).
$$
Here by \rp{intapp}, we have chosen $\{b_{\all\beta}^t\}\in C^{\infty,\infty}$
such that $|b^t_{\all\beta}(\zeta)-\f{\pd^2 r ^t}{\pd\zeta_\all\pd\beta_\beta}|
<\la_0/4$.
We have \re{elpp}.
We then define
\gan
f^t(z ,\zeta )=\begin{cases} \db_{z }(\chi (\zeta -z )\log F^t(z ,\zeta ))
      & \text{if $d/2<|\zeta -z |<d$}, \\
   0   & \text{otherwise}.
\end{cases}
\end{gather*}
As before, we can verify that $\{f^t\}$ is of class $C^{a+1,j}\bigl(\bigl\{D^t_{\e}\times
( D^t_{\del_1}\setminus \ov{D^t_{-\e}})\bigr\}\bigr)$.

Fix $t_0$. We apply the above to $D=D^{t_0}$ and denote $D_{\e'}$ by $D_{\e'}^{t_0}$, where $\e'$ is a regular value of $r^{t_0}$ and $\e/2<\e'<\e$.
Applying \rl{upss} to $K^t=\ov D^t\cup\{z\in\om^t\colon r^t\leq \e/2\}$, 
we conclude
that
\eq{DtDe}
D^t_{\e/2}\subset D_{\e'}^{t_0}\subset D_\e^t
\eeq
when $t$ is close to $t_0$. Here the second inclusion is verified directly.
Assume that $t$ is sufficiently close to $t_0$. Then we obtain
\begin{gather}
u^t(z ,\zeta )=T_{D^{t_0}_{\e'}}f^t(\cdot,\zeta )(z ),\quad \zeta \in D^t_{\del_0}\setminus D^t_{-\e},
\quad z \in D^t_{\e/2},\label{utzt}
\\
\Phi^t(z ,\zeta )=\begin{cases}
 F^t(z ,\zeta )e^{-u^t(z ,\zeta )}    & \text{if $|\zeta -z |<d$}, \\
e^{\chi(\zeta -z )\log F^t(z ,\zeta )-u^t(z ,\zeta )}    & \text{otherwise}.
\end{cases}\nonumber
\end{gather}
Since $\{ r ^t\}\in C^{a+2,j}(\{\ov{U^t}\})$, then $\{\pd_{\zeta }
 r ^t\}\in C^{a+1,j}(\{\ov{U^t}\})$.

We now use a partition $\{\chi_\nu\}$ of unity on $[0,1]$ such that each $\chi_\nu$  has
 compact support in a small neighborhood $\cL U(t_\nu)$
of $t_\nu$ and we define \re{utzt}.  We then define
\gan
\tilde u^t(z ,\zeta )=\sum\chi_\nu(t)T_{D_{\e'}^{t_\nu}}f^t(\cdot,\zeta )(z ),\\
\tilde\Phi^t(z ,\zeta )=\begin{cases}
 F^t(z ,\zeta )e^{-\tilde u^t(z ,\zeta )}    & \text{if $|\zeta -z |<d$}, \\
e^{ \chi(\zeta -z )\log F^t(z ,\zeta )-\tilde u^t(z ,\zeta )}    & \text{otherwise}.
\end{cases}
\end{gather*}
For $z \in\cL U(\ov {D^t})$, we have $\db\tilde u^t=\sum_\nu\chi_{\nu}(t)f^t(z ,\zeta )=
f^t(z ,\zeta )$ for $\db$ in $\ov{z }$.
Define $\Phi^t$, $M^t$ in the theorem to be  $\tilde\Phi^t$, $e^{-\tilde u^t}$ respectively.
Note that the integral operator $T_{D_{\e'}^{t_\nu}}$ is independent of $t$,   we can verify   $(iii)$.
 The proof for $C_*^{b+1,j}$ regularity is almost identical as the solution  operator  in \re{utzt} is independent of $t$. Thus we do not repeat the argument here.  \end{proof}

We need Oka-Hefer decomposition with multi parameters.  Let us first introduce the following.

\begin{defn}\label{oh5} Let $a,b\in\ov\rr_+$ and let $j\in\ov\nn$ with $j\leq a$.
Let   $ \{w^t\},  \{Y^t\}$ be continuous families of domains in $\cc^n$, $\rr^d$, respectively.
Denote by
 $
 C^{a,j}(\{\cL O(\om^t),Y^t\})
$
 the subspace of   $\{f^t\}\in C^{a,j}(\{\om^t\times Y^t \})$ satisfying  $f^t(\cdot, y)\in\cL O(\om^t)$ for all $y\in Y^t$.
We  define $C_*^{b,j}(\{\cL O(\om^t),Y^t \})$ analogously. \end{defn}

%

\le{pext}  Let $a\in \ov\rr_+$ and let $j\in\ov\nn$ with $j\leq a$. 
Let $\{D^t\}$ 
be a continuous family of  domains in $\cc^n$ with  total space $\cL D$.
 Let $\{\var^t\}\in C^{0,0}(\cL D)$ be a family of plurisubharmonic functions $\var^t$ on $D^t$.  Let
$$
\om^t_c=\{z\in D^t\colon\var^t(z)<c\}
$$
with total space $\om_c$.
Let $0<c_1<c_0$. Assume that
 $\om_{c_0}$ is relatively compact in $\cL D$.
    Let $H$ be a complex hyperplane in $\cc^n$. Assume     that
$
H\cap{\om_{c_1}^t}\neq\emptyset$, for all $ t\in[0,1].
$
There exits a linear continuous extension mapping
$$E\colon
C^{a,j}(\{\cL O(H\cap\om^t_{c_0}),Y^t\})\to C^{a,j}(\{\cL O(\om^t_{c_1}),Y^t\})$$
 such that  $(Ef)^t=f^{t}$ on ${(\om_{c_1}^t}\cap H)\times Y^t$.
 The similar conclusion holds when  $C^{a,j}$
  in hypotheses and the conclusion
   are replaced by $C^{b,j}_*$. 
\ele
\begin{proof} We may assume that $H$ is defined by $z_n=0$. Let $z=(z',z_n)$.
 Fix $t$. Since
 $\om^t_{c_0}$ is relatively compact in $D^t$, then
 $\om^t_{c_1}$ is relatively compact in $\om^t_{c_2}$ for $c_1<c_2\leq c_0$. Also, $\om^t_{c}$
 is a pseudoconvex domain for each $c<c_0$, if it is non-empty.

Fix $c_1<c_2<c_3<c_4<c_5<c_0$. Fix $t_0$.  When  $t$ is close to $t_0$, we have as \re{DtDe}
$$
\om_{c_1}^{t}\subset\subset \om_*\subset\subset \om_{c_2}^{t_0}  \subset \subset  \om_{c_3}^{t_0}\subset\subset \om_{c_4}^{t_0}\subset\subset \om_{c_5}^{t_0}\subset\subset  \om_{c_0}^{t}.
$$
Here $\om_*$ is   strictly pseudoconvex domain with  $C^2$ boundary  and it depends only
on $t_0$.

Note that $(\ov{\om_{c_4}^{t_0}}\cap H)\times\ov\Delta_{\e}\subset \om_{c_5}^{t_0}$ for some
$\e>0$. Let $\chi(z)$
be a smooth function that is equal to $1$ on $(\ov{\om_{c_3}^{t_0}}\cap H)\times\ov\Delta_{\e/2}$ and has
compact support in $({\om_{c_4}^{t_0}}\cap H)\times\Delta_{\e}$.  Consider
 $$
 g^{t}(z,y)= \begin{cases}
f^{t}(z',y) \db_z(\f{1}{z_n}\chi(z)), &\  z\in   ({\om_{c_4}^{t_0}}\cap H)\times\Delta_{\e}\setminus H\times\{0\},\\
0, &\  z\in   ({\om_{c_3}^{t_0}}\cap H)\times\Delta_{\e/2}\cup
\{\cc^n\setminus   ({\om_{c_4}^{t_0}}\cap H)\times\Delta_{\e}\}.
 \end{cases}
 $$
 When $\e$ is sufficiently small, the union of two sets in the above two formulae contains
 $ \ov{\om_{c_2}^{t_0}}$.
Thus we see that $\{g^{t}\}\in C^{a,j}(\{\om_{c_2}^t\times Y^t\})$.

 We now use the linear $\db$-solution operator $T_{\om_*}$ and define $u^{t}(z,y)=T_{\om_*}g^{t}(z,y)$. Then
  $\tilde f^{t}(z,y):=\chi(z)f^{t}(z',y)-z_nu^{t}(z,y)$ are holomorphic extensions respectively on
  $\om_{c_1}^t$  for $t$ close to $t_0$.  Clearly, for $t$ near $t_0$, we obtain that $\{u^{t}\}$
  and hence $\{\tilde f^{t}\}$ are in $C^{a,j}(\{\om_{c_1}^t\times Y^t\})$.  Furthermore, $\{\tilde f^{t}\}\in
  C^{a,j}(\{\cL O( \om_{c_1}^t),Y^t\})$.
  Using
 a partition of unity $\{\psi_\nu \}$ for $[0,1]$, we   get a desired holomorphic extension
 $$
 \tilde f^{t}(z,y)=\sum \psi_\nu(t)(\chi_\nu(z)f^{t}(z',y)-z_nu_\nu^{t}(z,y)).
 $$
 That is that, on $\om_{c_1}^t\cap H$, we have $\tilde f^{t}(z,y)=
 f^{t}(z',y).$
\end{proof}

Let $H$ be a complex subspace of $\cc^n$.
If $\{\om^t\}$ is a continuous family of domains $\om^t$ in $\cc^n$, let
$C_H^{b,i}(\{\cL O(\om_{c_0}^t),Y^t\})$ denote the space of $\{f^t\}\in
C^{b,i}(\{\cL O(\om_{c_0}^t),Y^t\})$ such that $f^t|_H=0$.
\le{} Let $D^t, \var^t,\om^t,\om^t_{c_0}$ be as in \rla{pext}.
Let $H$ be defined by $z_1=\cdots=z_\ell=0$. Let $0<c_1<c_0$.
Assume that
$
H\cap{\om_{c_1}^t}\neq\emptyset$, for all $t\in[0,1].
$
Suppose that $a'\leq a$, $i\leq j$, and $i\leq a$.
There exist linear maps
$$T_m\colon C_H^{a',i}(\{\cL O(\om_{c_0}^t),Y^t\})\to C^{a',i}(\{\cL O(\om_{c_1}^t),Y^t\})$$
 such that
$
f^{t}(z,y)=\sum_{m=1}^\ell T_m^{t}f(z,y)z_m.
$
The similar conclusion holds when $C^{a, j}$ is replaced $C_*^{b, j}$ for any $b\geq0$.
\ele
\begin{proof} Apply induction on $\ell$. For $\ell=1$, take $f_1^{t}(z,y)=f^{t}(z,y)/{z_1}$.  Away from $z_1=0$, $f_1$ clearly has
the desired smoothness.  To check the smoothness near $z_1=0$, we note that
$\ov{\Delta_\delta}\times
(\ov{\om_{c_1}^t}\cap H)\subset \om_{c_0}^{t_0}
$
if $\del$ and $|t-t_0|$ are sufficiently small.  For $z$ near $(0,p_2,\ldots, p_n)\in \ov{\om_{c_0}^t}\cap H$,
we have
$$
f_1^{t}(z,y)=\f{1}{(2\pi i)^n}\int_{|\zeta_1|=\del, |\zeta_2-p_2|=\del,\dots,
|\zeta_n-p_n|=\del}\frac{f^{t}(\zeta, y )}{\zeta_1\prod(\zeta_i-z_i)}\, d\zeta_1\cdots d\zeta_n.
$$
It is clear  that
$\{f_1^{t}\}\in C^{a,i}(\{\cL O(\om_{c_1}^t),Y^t\})$.

Assume that the lemma holds for $\ell-1$.  Then
$
f^{t}(0,z',y) =\sum_{j=2}^\ell g^{t}_j(z',y)z_j.
$
By \rl{pext}, we extend $\{g^{t}\}$ to $\{\tilde g^{t}\}\in C^{a,i}(\{\cL O(\om_{c_2}^t),Y^t\})$ such that  $\{\tilde g^{t}\}$ is
 still of class $C^{a,j}$.  Here $c_1<c_2<c_0$.
Define
$
\tilde f^{t}(y,z)=f^{t}(y,z)-\sum_{j=2}^\ell\tilde g_j^{t}(y,z)z_j.
$
Then $\tilde f^{t}(y,z)=0$ when $z_1=0$. So $\tilde f^{t}(y,z)=z_1\tilde g_1^{t}(y,z)$  and
$ \{ \tilde g_1^{t}\} \in C^{a,i}$.  %
%
%
%
\end{proof}

\begin{thm}[Hefer's theorem with multi-parameters]
\label{htmp}
Let $ r ^t, D^t,U^t$ be as in \rta{hrfn}.
There exists $c_0>0$ such that for $0<c_1<c_0$ there is a continuous linear map
$$
W\colon C^{a+1,j}\left(\{\cL O( D^t_{c_0}),Y^t\}\right)
\to C^{a+1,j}\left(\{\cL O( D^t_{c_1}\times D^t_{c_1}),Y^t \}\right)
$$
so that $w=Wf$ satisfies
$f^{t}(z,y)-
f^{t}(\zeta,y)=\sum_{i=1}^n w_i^{t}(z,\zeta ,y)(\zeta_i-z_i).
$
\eth
%
\begin{proof} On $D^t\times D^t\times Y^t$ we consider $F^{t}(z,\zeta,y )=f^{t}(z,y)-f^{t}(\zeta,y)$.
Now $(D_c^t)^2\subset(D^t\cup\om^t)^2$ is defined $\phi^t(z,w)<c$, where $\phi^t$ is the plurisubharmonic function in $(D^t\cup\om^t)^2$ defined by
$$
\phi^t(z,w):=\begin{cases}
\max\{-\del_*/2,r^t(z),r^t(w)\},& (z,w)\in (D^t\cup\om^t)^2\setminus ( D_{-\del_*})^2, \\
-\del_*/2, &(z,w)\in (D^t_{-\del_*})^2.
\end{cases}
$$
Here $\del_*$ is sufficiently small and $0<\del_*<\min(\del_1,\e)$ for $\del_1$ and $\e$ in \rt{hrfn}.
Apply the linear change of coordinates
\eq{Lpres}
L\colon \tilde z=z, \quad
\tilde \zeta=\zeta-z.
\eeq
Set $\tilde D^t=L(D^t\times D^t)$,   $\tilde f^{t}(\tilde z,\tilde\zeta,y)=f^{t}(\tilde z,\tilde\zeta+\tilde z,y))$, and
$\tilde\var^t(\tilde z,\tilde\zeta)=\phi^t(\tilde z,\tilde\zeta+\tilde z)$.  Now $L$, defined by \re{Lpres}
  identify  $C^{a+1,j}(\{ \cL O(D^t\times D^t),Y^t\})$ with $C^{a+1,j}(\{ \cL O(\tilde D^t\times D^t),Y^t\})$.
Replacing $D^t, \var^t,f^t$ by $\tilde D^t, \tilde\var^t$, $\tilde f^t$ respectively
in the last two lemmas, we obtain the conclusion.
 \end{proof}

\begin{thm}[Decomposition of Henkin-Ram{\'{\i}}rez functions] \label{rhfn}
Under the hypotheses of \rta{hrfn}, we have additionally
  $
\Phi^t(z,\zeta )= (\zeta -z )\cdot w^t(z,\zeta )
$
with  $$\{w_i^t\}\in C^{a+1,j}\left(\{\cL O(D^t_{\e/4}),D^t_{\del_1}\setminus D^t_{-\e}\}\right).$$
Analogously,  if
$\{ r ^t\}\in C_*^{b+2,j}(\{U^t\})$, then $\{w_i^t\}\in C_*^{b+1,j}(\cL O(D^t_{\e/4}),\{D^t_{\del_1}\setminus D^t_{-\e}\})$.
 \end{thm}
\begin{proof}Consider the family of holomorphic functions
$$
\{\Phi^t\}\in C^{a+1,j}\left(\{\cL O(D^t_{\e/2}),D^t_{\del_1}\setminus D^t_{-\e}\}\right).
$$
By \rt{htmp}, we have
$
\Phi^t(z,\zeta )-\Phi^t(\tilde z, \zeta)=\tilde w^{t}(z,\tilde z,\zeta)\cdot(\tilde z-z)
$
for $\zeta\in D^t_{\del_1}\setminus D^t_{-\e}$ and $\tilde z,z\in D^t_{\e/2}$.
Then $\tilde z=\zeta$ and $w^t(z,\zeta)=\tilde w^{t}(z,\zeta,\zeta)$ are the desired functions.
\end{proof}

The above theorem does not require that $\nabla r^t \neq0$ on $\pd D^t$, that is, $\pd D^t$
might not be smooth.  When in additionally $\pd D^t$ is smooth, we can use the Leray map $w(z,\zeta )$ to formulate homotopy formulae.  We take the Leray maps
$$
g^1(z,\zeta )=w^t(z,\zeta ), \quad g^0(z,\zeta )=\ov\zeta-\ov z.
$$
 We construct $\Om^0$, $\Om^{01}$ via $g^0$,  and
$g^0,g^{1}$ respectively.  The following theorem is a direct consequence of the classical homotopy formula
and the solution operator of Lieb-Range.
\th{phtf}
Let $a, b\in\ov\rr_+$ and let $j\in\ov\nn$
with $j\leq a$.  Let $\{U^t\}$ be a continuous family of domains
with a bounded total space $\cL U$, and let $\{\om^t\}$ be a continuous family of domains with $\om^t\subset U^t$.
 Let $\{ r ^t\}$ be of class $C^{a+2,j}(\cL U)$ $($reps. $C^{b+2,j}_*(\cL U))$.   Suppose that $ r ^t$
is strictly plurisubharmonic on  $\om^t$.
 Let $D^t$ be  relatively compact open sets in $U^t$.  Suppose that $\pd D^t=\{z\in\om^t\colon  r ^t(z)=0\}$, $ r ^t>0$ on $U^t\setminus\ov {D^t}$,
 and $\pd r ^t\neq0$ on $\pd D^t$.  Let $f^t$ be a $(0,q)$ form on $D^t$ of class $C^{1,0}(\cL D)\cap C^{0,0}(\ov{\cL D})$.
  Then
  \gan
 f^t=\db T_{q-1}^tf+T_{q}^t\db f,\quad 1\leq q\leq n,\\
T_{q-1}^tf=-\int_{\pd D^t}\Om_{0,q-1}^{0 {1}}\wedge f^t+\int_{D^t}\Om_{0,q-1}^{ 0}\wedge f^t.
\end{gather*}
Here $g^{1}=\Phi^t=w^t\cdot(\zeta-z)$ is given by \rta{rhfn} and $g^0=\ov\zeta-\ov z$.
Assume further that  $\db f^t=0$ on $D^t$. Then $\db S_q^tf=f$ for
\gan 
S^{t}_qf=L^t_qEf+K^t_q\db Ef, \quad 1\leq q\leq n,\\
L^t_qEf= \int_{\cL U^t }\Om^0_{0,q-1}\wedge E^tf,  \quad
\quad K^t_q\db Ef=
 \int_{\cL U^t\setminus D^t}\Om_{0,q-1}^{01}\wedge \db_\zeta E^tf.
 \nonumber
\end{gather*}
Here $\{E^tf\}$ is a Seeley extension such that for each $t$, $E^tf$ has compact support in $\cL U(\ov{D^t})$.
 \end{thm}

The proof for strictly convex case of \rp{estK} can be modified easily to obtain the following.
\th{5point9} Let $1\leq q\leq n$. Let $j,k\in\ov\nn$ with $j< k$.
  Let $\{D^t\}$, with total space $\cL D$,  be a smooth family of  strictly pseudoconvex domains $D^t$ with $C^{k+1,j}$ boundary. Then the $\db$-solution operators $S_{q}$ for $(0,q)$-forms on $D^t$ in \rta{phtf} satisfies
$$
\|S_qf\|_{\cL D; k+1/2,j}\leq C_{k}(\|r\|_{k+1,j})\|f\|_{\cL D;k,j}.
$$
Here the constant $C_k(\|r\|_{k+1},j)$ depends on norm of $\|r\|_{k+1,j}$.
\eth
Note that Theorems~\ref{phtf} and \ref{5point9} yield another proof the first part of \rt{dbartb}, which is  proved in \rt{spc}.

\setcounter{thm}{0}\setcounter{equation}{0}
\section{The $\db$-equation for variable domains of holomorphy}
\label{sect:ireg}

In this section, we present a parametrized version of Oka-Weil approximation in \rp{okwe}.  We then obtain interior regularity for a continuous family of domains of holomorphy that admits a continuous family of plurisubharmonic uniform exhaustion functions in \rt{intr}. As applications we  then solve a parametrized version of Levi-problem in \rt{levip}
and a parametrized version of Cousin problems  in \rt{cousin}. The formulation and the solutions of the Cousin problems   lead us to study functions defined on general open sets of $\cc^n\times[0,1]$ that are total spaces of
$\{D^t\}$ with $D^t$ being possibly empty.

It is a standard fact that for a domain $\Om$ in $\cc^n$ which is not the whole space,   $D$ is pseudoconvex if and only if $-\log\dist(z,\pd \Om)$ is a \psh\ function on $\Om$.
Therefore, we have the following.
\le{var0}
 Let
 $\Gaa^t\colon \ov{D}\to \ov{D^t}$ be a homeomorphism with $\Gaa^t(D)=D^t$ for each $t\in[0,1]$, where  $D,D^t$  are bounded domains in $\cc^n$. If $\Gaa\in C^{0,0}(\ov D)$ and  $D^t$ are domains of holomorphy, then
   $\{-\log\dist(z, 
\pd D^t)\}\in C^{0,0}(\cL D)$
 is a family of \psh\  uniform exhaustion functions on $D^t$.
\ele
Let  $\{D^t\}$ be a continuous family of domains in $\cc^n$, i.e. $\cL D=\cup D^t\times\{t\}$ is open
in $\cc^n\times[0,1]$.
If $K$ is a subset of $\cL D$, we define
$$
K^t=\{z\colon (z,t)\in K\}.
$$
By \rl{upss}, $K$ is compact,  if and only if $K^t$ are compact and
$\{K^t\}$ is {upper semi-continuous}.

Let $P^0(D)$ denote the set of continuous \psh\ functions on $D\subset\cc^n$, and let $P(D)$ denote the set of \psh\ functions on $D$.
For a subset $K$ of $D\subset\cc^n$, the $P(D)$  hull of $K$, denoted by $\widehat K_D^P$, is the set of $z\in D$ satisfying
$$
 \var(z)\leq\sup\{\var(w)\colon w\in E\}, \quad \forall \var\in P^0(D).
$$

\pr{smplur} Let  $\{D^t\}$ be a continuous family of non-empty domains in $\cc^n$ with total space $\cL D$.
Let  $\{K^t\}$ be a family of $($possibly empty$)$ subsets $K^t$ of $D^t$ with a compact total space
$\cL K$.
Let $\var^t$ be plurisubharmonic functions on $D^t$  with  $\{\var^t\}\in C^{0,0}(\cL D)$.
Let $\e>0$. There exists  a family  $\{\tilde\var^t\}\in C^{\infty,\infty}(\cL D)$ of functions such that $\tilde\var^t$  are plurisubharmonic on $\om^t$, the total space $\om$ of $\{\om^t\}$ is
an open neighborhood of $\cL K$,  and
$$
0\leq\tilde\var^t(z)-\var^t(z)<\e, \quad\forall z\in K^t.
$$
Here the last requirement is vacuous if $K^t$ is empty.
\epr
\begin{proof} To apply smoothing for the variable domains,  we first extend $\{D^t\}$ to a larger
family of domains as follows. We set $ D^t=D^0, K^t=K^0$ for $t<0$ and $ D^t=D^1, K^t=K^1$ for $t>1$.
We define  by
$$
\var^t_0(z):=\var^{0}(z), \quad t<0;  \quad \var^t_0(z):=\var^1_0(z), \quad t>1.
$$
Then  $\var^t_0$ is   plurisubharmonic on $D^t$.
Let $\chi(t)$ be a non-negative smooth function with compact support such that $\chi(t)=1$ near $0$ and $\int\chi(t)\, dt=1$. Let $\tilde\chi(z)=c_n\chi(|z|)$ such that $\int\tilde\chi(z)\, dV(z)=1$. Set
$
\chi_\del(t)=\del^{-1}\chi(\del^{-1}t)$ and $ \tilde\chi_\del(z)=\del^{-2n}\tilde\chi(\del^{-1} z).
$
 Consider
$$
\var_\del^t(z)=\iint\var_0^{s}(\zeta)\chi_{\del}(t-s)\tilde\chi_{\del}(z-\zeta)\, ds\, dV(\zeta).
$$
   Note that $\cL K$ is compact.
The sub-mean value property holds for $\var_\del^t$ near $K^t$ for all $t$ when $\del$
is sufficiently small. Therefore, each $\var_\del^t$ is
 plurisubharmonic near  $K^t$ for $0\leq t\leq1$. Using a partition of unity and further smoothing, we can achieve $\{\var^t\}\in C^{\infty,\infty}(\cL D)$.
\end{proof}

\pr{vart} Let $\{D^t\}$ with total space $\cL D$ be as in \rpa{smplur}.
 Suppose that there exist  \psh\    uniform exhaustion functions $\var_0^t$ on  $D^t$ with $\{\var_0^t\}\in C^{0,0}(\cL D)$.
 Let $\{K^t\}$ be a family of $($possibly empty$)$ sets with total space $\cL K$ being a
 compact subset of $\cL D$ and let $\om$ be an open neighborhood of $\cL K$ in $\cL D$. Assume that each  $K^t$ is ${P(D^t)}$ convex.
 There exist   \sps\    uniform exhaustion  functions $\var^t$ on $D^t$
 satisfying the following:
 \bppp
 \item $\{\var^t \}\in C^{\infty,\infty}(\cL D)$.
 \item $\var^t<0$ on $K^t$ and $\var^t>1$ on $D^t\setminus \om^t$.
 \eppp
\epr
\begin{proof} We first prove it when $K^t$ and $\om^t$ are empty for all $t$.
Since $\{(z,t)\colon\var_0^t(z)\leq0\}$ is compact, adding a constant we may assume that $\var_0^t\geq 2$ on $D^t$.
Consider sub-level sets
$$
 {E^t_m}   
=\{z\in D^t: \var_0^t(z)\leq m\}.
$$
   By assumption, the total space $\cL E_m$ of the family is compact.
By \rp{smplur}, we  find $\{\var_1^t\}\in C^{\infty,\infty}(\cL D)$
such that for each $t$, $0\leq\var_1^t-\varphi_0^t<1/4$ on  $E_2^t$
and $\var_1^t$ is  plurisubharmonic on $E_2^t$.
 Set $\hat\var_1^t=\var_1^t$.
Let $\chi_2 $ be a $C^\infty$ convex function such that $\chi_2(s)=0$ for $s<5/4$ and $\chi_2''(s)>0$ for $s>5/4$.
Take $\{\phi_2^t\}\in C^{\infty,\infty}$ such that $\phi_2^t$ is plurisubharmonic
 on $E_3^t$ and $0\leq\phi_2^t-\var_0^t<1/4$ on $E_3^t$.
Let $\hat\var_2^t=
\hat\var_1^t+C_2\chi_2 \circ\phi_2^t$.
 When $C_2$ is sufficiently large, $\hat\var_2^t$ is \sps\ on $E_3^t$ and
 $\hat\var_2^t>1$ on $E_3^t\setminus E_2^t$.
Inductively, we find   plurisubharmonic $\phi_m^t$ such that $| \phi_{m}^t -
\var_0^t|<1/4$ on $E_{m+1}^t$ with $\{\phi_m^t\}\in C^{\infty,\infty}$. Take a smooth convex function $\chi_m$ such that
$\chi_m(s)=0$ for $s<m+1/4$ and $\chi_m''(s)>0$ for $s>m+1/4$.
Choose $C_m$ sufficiently large such that
$\hat\var_{m}^t=
 \hat\var_{m-1}^t+C_m\chi_m\circ\phi_{m}^t$  is \psh\ on $E^t_{m+1}$ and
 $\hat\var_{m}^t>m$ on $E_{m+1}^t\setminus E_{m}^t$.
   Then $\hat\var^t=\lim_{m\to\infty}\hat\var^t_m$ satisfies
all the conditions.

We now deal with the general case.        Fix $t_0$. We first consider the case where   $K^{t_0}$ is empty.  By the upper semi-continuity of $\{K^t\}$, we know that $K^t$ is empty when $t$ is close to $t_0$.  So the above argument shows there are plurisubharmonic functions $\var^t$ defined for $t$
in a neighborhood of $t_0$ such that $\var^t>1$ on $D^t\setminus \om^t$ while the family
$\{\var^t\}$ has class $C^{\infty,\infty}$.
We  now consider the  case where   $K^{t_0}$ is non-empty.
Since $K^{t_0}$ is $P(D^{t_0})$ convex,  we can find
a continuous \psh\  function $\phi^{t_0}$ on $D^{t_0}$ such that
\ga\label{phi0}
\max_{K^{t_0}}\phi^{t_0}<0< 1<\min_{E_2^{t_0}\setminus \om^{t_0}}\phi^{t_0}.
\end{gather}
Let $L^{t_0}=\max_{E_2^{t_0}}\phi^{t_0}$. Then $L^{t_0}>0$. Define
$$
\var^{t}(w)=\begin{cases}
 \max(L^{t_0}\var^{t}_0(w),\phi^{t_0}(w))     & \text{if $w\in E_2^{t}$}, \\
L^{t_0} \var_0^t(w)     & \text{if $w\in D^t\setminus E_2^{t}$}.
\end{cases}
$$
Suppose that $t$ is close to $t_0$.
 We want to show that   $\var^{t}$ is well-defined and  \psh\  on $D^t$, and  $\{\var^t\}\in C^{0,0}(\cL D)$.
The $E_2^{t}$ is contained in $D^{t_0}$ since $\{E_2^t\}$ is upper semi-continuous by \rl{upss}.  Hence $\var^t$ is well-defined.
At $w\in\pd E_2^{t}$,  we have
$$
L^{t_0}\var_0^t(w)=2L^{t_0}> \phi^{t_0}(w).$$
This shows that $\{\var^t\}$ is in $C^{0,0}$ for $t$ close to $t_0$.  That $\var^t$ is   \psh\    on $D^{t}$ follows   from the definition and the inequality we just proved.
We can also see that $\{\var^t\}\in C^{0,0}$. By the upper semi-continuity of $\{K^t\}$ and
$\{E^t_2\setminus \om^t\}$, we conclude that
\eq{kt01}
\var^t<0 \quad \text{on $K^t$},  \quad \var^t>1\quad\text{on $E_2^t\setminus \om^t$}
\eeq
for $t$ close to $t_0$.  (Of course, the first inequality is vacuous if $K^t$ is empty.)
  By \rp{smplur}, we may assume that $\{\var^t\}\in C^{\infty,\infty}(\{D^t\})$.
 Again, $\{\var^t\}\in C^{0,0}$  is a family of \psh\ uniform exhaustion functions on $D^t$.

We now find a finite open covering $\{I_\all; \all \in A\}$ of $[0,1]$,  and continuous \psh\ functions $\var_\all^t$ such that the above holds for $\var^t=\var_\all^t$ and $t\in I_\all$.
Let $\{\chi_\all\}$ be a partition of unity subordinate to $\{I_\all\}$ with  $ \chi_\all\geq0$.  Then $\var^t=\sum\chi_\all(t)\var_\all^t$ is \psh\ on $D^t$ and \re{kt01} holds
for all $t$.  We can repeat the above smoothing and approximation arguments for $\var^t$.  The proof is complete.
\end{proof}

 We need the following parametrized version of the Oka-Weil approximation theorem, which
 is crucial in solving the $\db$-equation for variable domains of holomorphy.  Let us first
 state the following elementary result which follows from the Morse lemma.
\le{morse} Let $r$ be a real function on a domain $D$ in $\rr^N$ that
 has no degenerate critical point. Assume that $K$ is a compact subset of $D$ such that $r$ does not attain any  local maximum value at each point in $K$.
For any $\del>0$ there exists $\e>0$ such that for each $z\in K$ there exists $\zeta$ such
that
$$
r(\zeta)-r(z)>\e, \quad |\zeta-z|<\del.
$$
\ele
In other words, the lemma concludes that the value of $r$ must increase  by a fixed amount from a point in $K$ within a fixed distance.
\pr{okwe} Let $\{D^t\}$ and $\cL D$ be as in \rpa{smplur}. Suppose that  there are
 plurisubharmonic  uniform exhaustion functions $\var^t$ on $D^t$  with $\{\var^t\}\in C^{0,0}(\cL D)$.  Fix $\ell\leq j$ and $\ell<\infty$.
Let $\{\om^t\}$ a continuous family of domains $\om^t$ such that each $\om^t$ contains $K_0^t$ for
 $
K_c^t:=\{z\in D^t\colon\var^t(z)\leq c\}.
$
Let  $\{f^t\}\in C^{0,j}_*(\{\om^t\})$ be a family of holomorphic functions $f^t$ on $\om^t$.
Let $\e>0$. There are  holomorphic functions $g^t$ on $D^t$
 such that  $\{g^t\}\in C_*^{0,j}(\cL D)$ and
$$
|g-f|_{\cL K_0;0,\ell}:=\max_{i\leq \ell,t\in[0,1], K_0^t\neq\emptyset}\,\max_{z\in K_0^t}|\pd_t^i\{(g^t-f^t)(z)\}|<\e.
$$
\epr
\begin{proof} We follow the method of approximation by using Leray formula.
We say that assertion $A_c(\{f^t\})$ holds   if  for each  $\e>0$ and
  each finite $\ell\leq j$
 there exists a family of functions $g^t$ satisfying the following:
 \bppp
 \item There is a family $\{g^t\}\in C_*^{0,j}(\{{K}_c^t\})$ of holomorphic functions $g^t$ defined near $K^t_c$, i.e. $g^t\in\cL O(K_c^t)$.
 \item
 $|g-f|_{\cL K_0;0,\ell} 
 <\e.$
 \eppp
  By assumptions, $A_0$ holds for all $\{f^t\}$. Let $c_*$ be the supremum of $c$ such that    $A_c$
  holds for all $\{f^t\}$.

  Let us  first show that $c_*=\infty$.
Assume for the sake of contradiction that $c_*$ is finite.
 By assumption, we know that $c_*>0$.
We may further assume that
\eq{csmin}
f^t\in \cL O(K^t_{95c_*/100}), \quad
\{f^t\}\in C^{0,j}_*(\{K^t_{95c_*/100}\}).
\eeq

We  find a finite open covering $\{I_i\}_{i=1}^m$ of $[0,1]$
and $t_i\in I_i$  and $\del_i>0$ which have the following properties:
For   real linear functions $L_i$ with $|L_i|<\f{1}{200}\min\{1,c_*\}$ on $K_{c_*+9}^{t_i}$, define
\gan
r_i:=\var^{t_i}*\chi_{\del_i}+L_i\in C^\infty(K_{c_*+9}^{t_i}), \\
\max\{|r_i(z)-\var^{t_i}(z)|\colon z\in K_{c_*+9}^{t_i}\}<\f{1}{100}\min\{1,\e_*\}.
\nonumber
\end{gather*}
  Here $\chi_{\del}(z)=\del^{-2n}\chi(\del^{-2n}z)$,
  $\chi\geq0$ is a smooth function
with compact support in $\cc^n$ , and $\int\chi =1$.
For suitable $\del_i>0$ and $L_i$,  the $r_i$
 are strictly plurisubharmonic and have no degenerate  critical point on $K^{t_i}_{c_*+9}$.
 Furthermore,  if
 $$
\Om^i_{c}:=\{z\in D^{t_i}\colon r_i(z)<c\}, \quad E^i_c:=\{z\in D^{t_i}\colon r_i=c\}, \quad L^i_c=\Om^i_c\cup E^i_c,
$$
and    $t\in I_i$, then
\eq{LtLi}
K_c^t\subset \Om^i_{c+c_*/100}, \qquad  L^i_c\subset K_{c+c_*/100}^t, \quad \forall c\leq c_*+8.
\eeq

 Using these properties, we choose $d_0>0$ such that
 $$
| \pd^2_\zeta r_i-\pd^2_zr_i|<\f{\la_0}{C_n}, \quad\text{for $ |\zeta-z|<d_0$}.
 $$
 We know that there exists $\la_0>0$ such that
 $
 \sum_{j,k}\DD{^2r_i}{\zeta_j\pd\ov\zeta_k}s_j\ov s_k\geq\la_0|s|^2.
 $
 Define  \gan
 d=\min\left\{d_0,\dist(L^i_{c+1}, E^i_{c_*+9})\colon 8c_*/10\leq c\leq c_*+6, 1\leq i\leq m\right\}.
 \end{gather*}
 Then $d>0$.
We now  apply  Theorems~\ref{hrfn} and \ref{rhfn} for a fixed domain as follows. For all $t$ we take $\om^t=\Om^i_{c_*+9}$,  $ r ^t=r_i-c$, and $\del_0=1$.
By \re{dd0e}, we find $\e_0\in(0,1)$ such that for any $c$ satisfying $8c_*/10\leq c\leq c_*+6$,   there are functions $\Phi_i(z,\zeta )$
 and mappings $w_i(z,\zeta )$ that  satisfy the properties:
 \bppp
\item  Let $z\in \Om^i_{c+\e_0/4}$ and  $\zeta\in \Om^i_{c+3} \setminus \Om^i_{c-\e_0}$. Then
  $\Phi_i(z,\zeta )$, $w_i(z,\zeta )$
   are holomorphic in $z$ and $C^1$ in $z,\zeta$, and
 $
 \Phi_i(z,\zeta )=w_i(z,\zeta )\cdot(\zeta-z).
 $
 \item If  $\pd \Om^i_c\in C^1$ and $h\in A(\ov{ \Om^i_c})$,  then
  \eq{lerayf}
 h(z)=c_n\int_{\pd \Om^i_c}
 h(\zeta)\f{w_i(z,\zeta )\cdot d\zeta\wedge(\db_\zeta
 w_i(z,\zeta )\cdot
 d\zeta))^{n-1}}{\Phi_i^n(z,\zeta )}, \quad\forall z\in \Om^i_c.
 \eeq
 \eppp
 It is crucial from \rt{rhfn} that although $\Phi_i$, $w_i$ depend on $c$, the $\e_0$  does not depend
 on $c$. Also, as mentioned in \rrem{2leray},  we  need the Leray formula \re{lerayf} to carry out the following approximation.

By \re{elpp} we obtain
 $
| \Phi_i(z,\zeta )|\geq C
$
for $r_i(\zeta)\geq c_*/2$ and $r_i(z)\leq c_*/100$.
Choose $d_1\in(0,d)$ such that if $|\tilde\zeta-\zeta|\leq d_1$, $r_i(z)\leq c_*/100$,
and  $r_i(\tilde\zeta)\geq r_i(\zeta)\geq c_*/2$, then
$$
|\Phi_i(z,\tilde\zeta)-\Phi_i(z,\zeta )|\leq \f{1}{4}|\Phi_i(z,\tilde\zeta)|.
$$
This shows that $\f{1}{\Phi_i(z,\zeta )}$, which is holomorphic in $z$ for $r_i(z)<r_i(\zeta)$,
can be approximated by  polynomials in $\f{\Phi_i(z,\zeta )-\Phi_i(z,\tilde\zeta)}{\Phi_i(z,\tilde\zeta)}$
on $r_i(z)\leq c_*/100$ in the super norm.  The quotient and $w_i(z,\zeta )$
are holomorphic in $z$ on the domain defined by
$$
r_i(z)<r_i(\tilde \zeta), \quad r_i(z)<r_i(\zeta)+\e_0/4.
$$
Since $r$ is strictly plurisubharmonic in $D_{c_*+9}^{t_i}$,  it does not attain any local maximum value
in the set.
By \rl{morse},  there exists $\e_1>0$, depending on $d_1$, such that  if $r_i(\zeta)\leq c_*+8$,
there exists $\tilde\zeta$ such that
\eq{rijump}
r_i(\tilde \zeta)\geq r_i(\zeta)+\e_1, \quad  |\tilde\zeta-\zeta|\leq d_1.
\eeq
Let $\e_2=\min(\e_0/4,\e_1, \f{1}{2})$.
Therefore,  for each $\zeta$, $\f{1}{\Phi_i(z,\zeta )}$
can be approximated
on the domain defined by $r_i(z)\leq c_*/100$ by     holomorphic functions on $r_i(z)<r_i(\zeta)+\e_2$.
This is the approximation that we will use in the following argument.

Fix $i$ and assume that $t\in I_i$.
We start with a regular value $\hat c_1$ of  $r_i$ with $\hat c_1\in (93c_*/100, 94c_*/100)$.
By \re{csmin}, we can replace $h,c$   in \re{lerayf} by $f^t$ ($t\in I_i$),
${\hat c_1}$ respectively.
This gives us an integral representation for $f^t$. We then differentiate $f^t$ and the integrand
to get integral representations for $\pd_t^lf^t$ for $l\leq \ell$.  For the $\ell+1$ integrals, we approximate them by Riemann sums to
 obtain for a given $\e>0$,
 \ga \label{ftg1t}
\sup_{t\in I_i}\sup_{z\in K_0^t} \left|\pd_t^l\left\{
f^t(z)-g_i^t(z)\right\}
 \right|<\e,\\
 g_i^t(z)=\sum_{m=1}^N \Bigl\{c_nf^t(\zeta)\f{P(w_i(z,\zeta ), \pd_{\ov\zeta} w_i(z,\zeta ))}
 {(\Phi_i)^n( z,\zeta )}\Bigr\} \bigg|_{\zeta=\zeta_{i,m}}.
 \nonumber
 \end{gather}
 Here $P$, $\zeta_{i,m}$ are independent of $t\in I_i$,
 $l=0,\dots, \ell$, $P$ is a polynomial, and
  $\zeta_{i,m}\in E^i_{\hat c_1}$.  By the approximation  obtained earlier, we know
  that each term in the above sum of $g_i^t$ can be approximated on $K_0^t$ by
    functions  that are
  holomorphic on
  $
  L^i_{\hat c_1+\e_2}.
  $
  This gives us a holomorphic function $\tilde g_i^t$  on the above set such that \re{ftg1t} holds
  when $g_i^t$ is replaced by $\tilde g_i^t$.
  Next, we choose a regular value $\hat c_2$  of $r_i$ with $\hat c_1+\e_2/2<\hat c_2<\hat c_1+\e_2$. We repeat the argument.
 We repeat this $n_i$ times, with $n_i$ independent of $t$ via \re{rijump},  to obtain \re{ftg1t} for a function $g_i^t$ that is holomorphic on $L^i_{\hat c_1+n_i\e_2/2}$. Here $n_i$ is so chosen that $c_*+5<\hat c_1+n_i\e_2/2<c_*+6$.
 From \re{LtLi} it follows that $g_i^t$ is holomorphic on
 $
 L_{c_*+4}^t$ for $t\in I_i.
 $
 Note that $L_{c_*+6}\supset K_{c_*+5}^t$.
 Using a partition of unity $\{\chi_i(t)\}$, we get the approximation $g^t=\chi_i g_i^t$  so that
 $g^t$ is holomorphic on $K^t_{c_*+5}$ and
 $
  \left\|
f  -g
 \right\|_{\cL K_0; 0,\ell}<\e.
  $
  This shows that $c_*=\infty$.

 To finish the poof,  let $\e>0$ and $\{h_i^t\}\in C^{0,j}_*(\{K_i^t\})$ such that $h_i^t\in \cL O(K_i^t)$ and
$$
\|f-h_i\|_{\cL K_0; 0,j}<\f{\e}{2}, \quad
\|h_{i+1}-h_i\|_{\cL K_i; 0,j}<\f{\e}{2^{i+1}}, \quad i=0,1,\ldots.
$$
Then $\lim_{i\to\infty} h_i^t$ has the desired properties.
 \end{proof}

\begin{rem}
By the Cauchy inequality, if $\tilde h^t$ approximates $h^t$ on $\om^t$, where $\om^t$ are open neighborhoods of  compact subsets $K^t$  in $C_*^{0,j}$ norm, then it also approximates in $C^{\infty,j}$.
 \end{rem}

We now prove the second part of \rt{dbartb}. Here we study the interior regularities. By avoiding the boundary regularities, we are able to use  spaces $C_*^{\ell+1+\all,j}$.
\th{intr} Let  $\{D^t\}$ be a continuous family of  non-empty   domains in $\cc^n$. Let $\{\var^t\}\in C^{0,0}(\cL D)$ be a family
 of plurisubharmonic   uniform exhaustion functions $\var^t$ on $D^t$.
 Let
 $0<\alpha<1$ and $1\leq q\leq n$,  and let $k,\ell, j\in\ov\nn$ with $k\geq j$. Let
  $\{f^t \}\in C_*^{\ell+\all,j}(\cL D)$ $($resp. $C^{k+\all,j}(\cL D))$ be a family of $\dbar$-closed
  $(0,q)$-forms on $D^t$.
There exist a family of  solutions $u^t$ to $\db u^t=f^t$
on $D^t$  so that $\{u^t\}$ is in $C^{\ell+1+\all,j}_*(\cL D)$ $($resp. $C^{k+1+\all,j}(\cL D))$.
\eth
\begin{proof}
Let $K_m^t=\{z\in D^t\colon\var^t(z)\leq m\}$. Denote by $\cL K_m$
the total space of $\{K_m^t\}$. Since $\{\var^t\}$ is in $C^{0,0}$ and each $D^t$ is non empty, there is a $c_0$ such that $K_{c_0}^t$ is non-empty for all $t$. Without loss of generality, we assume that all $K_{-1}^t$ are non-empty.
We will
 first find $u_m^t$  such
that $\db u_m^t=f^t$ near ${K_m^t}$; more precisely $\db u_m^t=f^t$ on $\tilde\om^t$ while $\{\tilde\om^t\}$ is a continuous family of domains of which the total space contains $\cL K_m$.
Furthermore,  $\{u_m^t\}$ is of class $C_*^{\ell+1+\all,j}(\cL D)$. Fix $t_0$ and assume that $t$ is sufficiently close to $t_0$. We know that $K_{m+1}^t$
contains $\ov {K_m^{t_0}}$. Take $D^*$ such that $D^*$ has a $C^\infty$ \spc\  boundary. Moreover,
$K_{m+1}^t$ contains $\ov{ D^*}$  and $D^*$ contains $\ov {K^t_m}$.  Let $T_{D^*}$ be a solution operator via the homotopy formula.
Let $u_m^t=T_{D^*}f^t$. Then $\{u_m^t\}$ is in $C_*^{\ell+1+\all,j}(\{K_m^t\})$ for $t$ close to $t_0$.  Using a partition of unity $\{\chi_i\}$ on $[0,1]$,
we can find $u_m^t$  such that for each $t$,
$\db u_m^t=f^t$ near $\ov{K_m^t}$. Using cut-off, we may further achieve $\{u^t_m \}\in C_*^{\ell+1+\all,j}(\cL D)$.

We now assume that $q>1$. Then $\db( u_1^t-u_2^t)=0$ near $\ov {K_1^t}$.  By the above arguments, we can find $v_1^t$ such that $\db v_1^t=u_1^t-u_2^t$
near $\ov {K_1^t}$ and $\{\db v_1^t\}\in C_*^{\ell+1+\all,j}(\cL D)$.   We take $\hat u_1^t=u_1^t$ and $\hat u_2^t=u_2^t+\db v_1^t$. Then $\hat u_2^t=\hat u_1^t$ near $\ov{ K_1^t}$, $\{\hat u_2^t\}\in C_*^{\ell+1+\all,j}(\cL D)$,
 and $\db\hat u_2^t=\db u_2^t=f^t$ near  $\ov {K_2^t}$.  Inductively, we have
$\db(\hat u_j^t-u_{j+1}^t)=0$ near $\ov{ D_j^t}$.  We find $v_{j}^t$ such that $\db v_{j}^t=\hat u_j^t-u_{j+1}^t$ near $\ov{K_j^t}$, and $\{\db v_{j}^t\}\in C^{\ell+1+\all,j}_*(\cL D)$.  Set $\hat u_{j+1}^t=u_{j+1}^t+\db v_{j}^t$. Then $\hat u_{j+1}^t=\hat u_j^t$ near $\ov{ K_j^t}$.

Assume now that $q=1$.  If $j$ is finite, we take $j_\ell=j$; otherwise, we take a sequence of integers $j_\ell$ tending to $\infty$. Analogously, we take a sequence of integers $k_\ell$ tending to $k$.
 Let $\hat u_1^t=u_1^t$.  Then $u_2^t-u_1^t$ is holomorphic near $\ov{K_1^t}$.
By \rp{okwe}, we can find holomorphic functions $h^t_1$ on $D^t$ such that $\{h_1^t \}\in C_*^{\infty,j}$ and
$$
|\hat u_1-u_2-h_1 |_{C_*^{\ell_1+1+\all,j_1}(\cL K_1)}<1/2.
$$
Let $\hat u_2^t=u_2^t+h^t$. We still have $\db \hat u_2=\db u_2^t=f^t$. Inductively, we find   holomorphic functions $h_\ell^t$ on $D^t$
such that  $\{h_\ell^t\}\in C_*^{\infty,j}$ and
$$
|\hat u_m-u_{m+1}-h_m |_{C_*^{\ell_m+1+\all, j_\ell}(\cL K_m)}<\f{1}{2^m}.
$$
Here $
\cL K_m=\{(z,t)\colon \var^t (z)\leq m\}.
$
We then define $\hat u_{m+1}=u_{m+1}+h_m$.  Using the Cauchy estimates, we verify that $\{\hat u_{m}^t\}$ converges to $\{\hat u^t\}$ in $C_*^{\ell_m+1+\all,j}(\cL D)$.

 Analogously, we can verify  $C^{k+1+\all, j}(\cL D)$ regularity of the  solutions.
\end{proof}

%

\begin{defn}  Let  $\{D^t\}$ be a continuous family of non-empty  domains in $\cc^n$.
 Let $A_{j}(\cL D)$ denote the set of families $\{f^t\}$ of holomorphic functions $f^t$ on $D^t$ with  $\{f^t\}\in C_*^{0,j}(\cL D)$.  Let $\{E^t\}$, with total space $\cL E$,  be a family of subsets $E^t$ of $D^t$, the $A^j(\cL D)$-hull $\{\widehat E^t\}$ of $\{E^t\}$ is defined by its total space
 $$
 \widehat{\cL E}=\left\{(z,t)\in \cL D\colon|f^s(z)|\leq\sup_{(w,s)\in \cL E}|f^{s}(w)|, \forall \{f^t\}\in
 A^j(\cL D)\right\}.
 $$
 We say that $\{E^t\}$ is $A^j(\cL D)$ convex if $\widehat E^t =E^t$ for all $t$.
 \end{defn}

 As an application of \rp{okwe}, we solve the following version  of the Levi problem for domains
 with parameter.
 \th{levip} Let  $\{D^t\}$ be a continuous family of non-empty  domains in $\cc^n$.
Let $\{\var^t\}\in C^{0,0}(\cL D)$ be a family
 of
 plurisubharmonic  uniform exhaustion functions $\var^t$ on $D^t$.
  Then $\{K^t_c\}$  is $A^j(\cL D)$   convex for all $c\in\rr$, where $K_c^t\subset D^t$ is defined by $\var^t\leq c$.
    \eth
  \begin{proof} By \rp{vart},  we may assume that $\{\var^t\}\in C^{\infty,\infty}(\cL D)$ and   $\var^t$ are strictly \psh\  on $D^t$.
    It suffices to show that
  $(\widehat {K_c})^t$ is contained in $K^t_{c_1}$, if $c_1>c$.

   Fix $p\in D^s$ with $\var^{s}(p)=c_1>c$.  Choose a real linear function $L$ such
  that $\phi^{s}=\var^s+L$ has only discrete critical points in $D^s$. Let $\tilde D_c^s$ be defined by $\phi^s<c$.
  We also choose $L$ so small on $K_{c_2}^s$   that for  some $c'>c$,
the $\tilde D_{c'}^s$ contains $K_c^s$.

   Let $F(w)$ be the Levi polynomial of $\var^{s}$ at $w_0=p^{s}$. Then
  $$
  \var^{t_0}(w)>\var^{t_0}(w_0)+2\RE F(w)+c|w-w_0|^2.
  $$
   Choose $\del>0$ sufficiently small. For $\var^{t_0}(w)\leq c'+\del$ and $|w-w_0|\geq \e$, we have
 \eq{2Ref}
 2\RE F(w)<\del-c\e^2.
 \eeq

  Let $\chi(w)$ be a smooth function supported in $B_{2\e}$ such that $\chi=1$ on $B_\e$. We consider
  $$
  u(w)=\chi(w-w_0) e^{L F(w)}-v(w).
  $$
  We want $u$ to be holomorphic on $K^{s}_{c'+\del}$ by solving
  $
  \db v_L=g_L(w):=\db(\chi(w-w_0)e^{LF(w)}).
  $
We have $g_L(w)=e^{LF(w)}\db\chi(w-w_0)$, which is zero on $B_\e$.   Thus \re{2Ref} implies that
  $$
|g_L|_{0}:= \sup_{K_{c'+\del}^{s}} |g_L|\leq Ce^{L(\del-2\e^2)}.
  $$
  Take $0<\del<c\e^2/2$ such that $D_*:=K^{s}_{c'+\del}$ has smooth boundary. Thus $|g_L|_{0}\to0$ as $L\to+\infty$. We now solve the $\db$-equation on $D_*$ by using
  the homotopy formula on $D_*$ to get the estimate
  $$
  |v_L|_0\leq C'|g_L|_{0},
   $$
   which is uniform in $L$. Fix $L$ such that $C'C e^{L(\del-2\e^2)}<1/4$.  On $K^{s}_c$, we get $|u|<1/4$. Also $|u(w_0)|\geq 1-|v_L(w_0)|>3/4$.  Note that $u$ is holomorphic on $K_{c'+\del}^{s}$.
  Fix $c<c''<c'+\del/2$.  Using a cut-off function $\chi(t)$ such that $0\leq\chi\leq1$ and $\chi(s)=1$.  We obtain $\{u^t=\chi(t)u\}$ such that each $u^t$ is holomorphic on $K_{c''}^t$. Moreover
   $$
   |u^t|<1/4\quad \text{on $K_{c}^t$}, \quad |u^{s}(w_0)|>3/4.
   $$
   By \rp{okwe}, we can replace $u^t$ by $\tilde u^t$ which is holomorphic on $D^t$ such that the above still holds for $\tilde u^t$, while $\{\tilde u^t \}\in A^{j}(\cL D)$.
   Therefore, $(p,s)$ is not in the $A^j(\cL D)$ hull of ${K_c}$.
  \end{proof}

We now consider Cousin problems with parameter.
We formulate the problems and its solution as follows. Here it is  more convenient to formulate the problems and find the solutions, by identifying a continuous family $\{D^t\}$ with its open total space $\cL D$
in $\cc^n\times[0,1]$.
 We also identity a family $\{f^t\}$ of functions $f^t$ on $D^t$ with a function $(x,t)\to f^t(x)$ on $\cL D$.
\th{cousin}
Let $0<\all<1$.  Let $j,k\in\ov\nn$ with $k\geq j$ and $k>0$. 
 Let $\{D^t\}$
be a continuous family of  non-empty  domains in $\cc^n$.  Suppose that   $D^t$ admit
 plurisubharmonic  uniform exhaustion functions $\var^t$ with $\{\var^t\}\in C^{0,0}(\cL D)$.  Suppose that $\{\cL D_a\colon a\in A\}$ is an open covering of $\cL D$, where $\cL D_a$ is the total space of $\{D_a^t\}$.
Let  $\{f_{ab}^t\}$ be a family of functions such that each $f^t_{ab}$ is holomorphic  on   $ D_a^t\cap D_b^t$. Assume that $\{f_{ab}^t\}\in C_*^{0,j}( \cL D_a\cap \cL D_b)$.
\bppp
\item $($First Cousin problem$.)$
Assume that for all $a,b,c\in A$,  $f_{ab}^t=-f_{ba}^t$,  and
$$
f_{ab}^t+f_{bc}^t=f_{ac}^t, \quad\text{on $ D_a^t\cap D_b^t\cap D_c^t$}.
$$
There exist
 families $\{f_a^t\}\in C^{\infty,j}( \cL D_a)$  of holomorphic functions    $f_a^t$
 on $ D_a^t$ such that
  $f^t_a-f^t_b=f^t_{ab}$.
  \item $($Second Cousin problem$.)$
   Suppose that  each $f_{ab}^t$
  does not vanish   on
  $ D_a^t\cap D_b^t$. Suppose that for all $a,b,c$, $f_{ab}^t=(f_{ba}^t)^{-1}$, and
  $$
  f_{ab}^tf_{bc}^t=f_{ac}^t, \quad\text{on $ D_a^t\cap D_b^t\cap D_c^t$}.
  $$
   There exists a family $\{f_a^t\}$ of nowhere vanishing holomorphic functions
  $f_a^t$
   on $ D^t_a$ such that  $\{f_a^t\}\in
  C^{\infty,j}(\cL D_a)$  and
  $
  f_a^t(f_b^t)^{-1}=f_{ab}^t,$
  provided  there exists a family $\{g_a^t\}\in C^{0,0}(\cL D_a)$ of functions $g_a^t$ vanishing
nowhere   on $D_a^t$ such that   $g_a^t(g_b^t)^{-1}=f_{ab}^t$ for all $a,b$.
  \eppp
  \eth
\begin{proof} Although all $D^t$ are non-empty,  we allow some $D_a^t$ to be empty for a non-empty total set $\cL D_a$.  Since each $\cL D_a$ is open in $\cL D$, then if  all $f^t_{a,b}$ are holomorphic on
$D_a^t\cap D_b^t$, the Cauchy formula implies that  $\{f_{ab}^t\}\in C^{\infty ,j}(  \cL D_a\cap \cL D_b)$
for $\{f_{ab}^t\}\in C_*^{0,j}( \cL D_a\cap \cL D_b)$.

 First, we consider the first Cousin problem.  
By assumption,  $\{\cL D_a\}$ is an open covering of $\cL D$.   We choose a partition of unity $\{\var_\nu\}$
that is  subordinate to the covering $\{ \cL D_a\}$.
 More precisely,    $\var_\nu\in
 C_0^\infty(\cL D_{e_\nu})$, all $\var_\nu$ but finitely many of them, vanish identically on any compact subset of $\cL D$,
and $\sum\var_\nu=1$ on $ \cL D$.  Set $\var_\nu^t(z)=\var_\nu(z,t)$ and
$$
g_a^t(z)=\sum \var_{\nu}^t(z)f_{e_\nu a}^t(z), \quad\forall  z\in D_a^t.
$$
Here $\var_{\nu}^t(z)f_{e_\nu a}^t(z)=0$ if $z$ is not in $D_{e_{\nu}}^t\cap D_a^t$.
%
 We can verify that $\{g_a^t\}\in C^{\infty,j}(\cL D_a)$ and $g_a^t-g_b^t=f_{ab}^t$. Thus $\db g_a^t=\db g_b^t$
   on $D_a^t\cap D_b^t$.  This shows
   that $\tilde g^t:=\db g_a^t$ is well-defined on $D^t$. Also $\{\tilde g^t\}\in C^{\infty,j}(\cL D)$.  By \rt{intr} we can find $\{u^t\}\in C^{\infty,j}(\cL D)$ such that
   $\db g_a^t=\db u^t$. Now $f^t_a=g_a^t-u^t$ becomes holomorphic on $D_a^t$ and $\{f_a^t\}$ is of class $C^{\infty,j}(\cL D_a)$, while $f_a^t-f_b^t=f_{ab}^t$.

 For the second Cousin problem, we assume that there exists $\{g_a^t\}\in C^{0,0}(\cL D_a)$ such that
 $
 f_{ab}^t=g_a^t(g_b^t)^{-1}.
 $
We first consider the case that  each $\mathcal D_a$ is simply connected. We can write $g^t_a=e^{\log g^t_a}$ with $\{\log g^t_a\}\in
C^{0,0}(\cL D_a)$. Let $h_{ab}^t=\log g^t_a-\log g^t_b$. We want to show that $\{h_{ab}^t\}$ is in $C^{\infty,j}(\cL D_a\cap \cL D_b)$.
Indeed,
$
e^{h_{ab}^t}=f_{ab}^t.
$
Locally, we get $h^t_{ab}(z)=\log f_{ab}^t(z)+2\pi i m(z,t)$ with $m(z,t)\in\zz$. By  continuity, we conclude that $m$ is locally constant.
This shows that $\{h^t_{ab}\}$ is in $C^{\infty,j}(\cL D_a\cap \cL D_b)$. By the solution of the first Cousin problem, we find $\{h_a^t\}\in
C^{\infty,j}(\cL D_a)$ such that $h_a^t$ is holomorphic on $D_a^t$,   $e^{h_a^t-h_b^t}=f_{ab}^t$, and
\eq{logc}
h_a^t-h_b^t=\log g_a^t-\log g_b^t.
\eeq
When  $\cL D_a$ are not simply connected, we apply a refinement$\{\tilde {\cL D}_\beta\colon \beta\in B\}$ to the open covering
of $\{\cL D_a\colon a\in A\}$ such that each $\tilde {\cL D}_\beta$ is a simply connected open subset of some $  {\cL D}_a$, while $\cup\tilde {\cL D}_\beta=\cL D$. By \re{logc}, we find $\{\tilde f_\beta^t\}\in C^{\infty,j}$  such that, for $\tilde{\cL D}_\all\subset\cL D_a$,
$$
h_\alpha^t-h_\beta^t=\log g_\all^t-\log g_\beta^t, \quad \log g_\all=\log g_a|_{\tilde{\cL D}_\all}.
$$
 Then $h_a^t=h_\alpha^t$ is well-defined on $D_a^t$ and $\{h_a^t\}$
provides the desired solutions.
\end{proof}

Analogously, we  verify the following result.
\th{cousinb}Let $0<\all<1$. Let   $j,k\in\ov\nn$ with $k-1\geq j$.    Let $\{D^t\}$
be a continuous family of    bounded strongly pseudoconvex domains in $\cc^n$ of   $C^{k+1,j}$ boundary.   Suppose that $\{ {\cL  D}_a  \colon a\in A\}$ is an open covering of   $\ov {\cL D}$, where $ {\cL D}_a$ is the total space of $\{  D_a^t\}$. Let  $\{f_{ab}^t\}\in  C^{k+1/2,j}(\cL D_a\cap   \cL D_b)$ be a family of functions  $f^t_{ab}$ that are holomorphic  in the interior of $  D_a^t\cap   D_b^t$ and satisfy  \rta{cousin} $(i)$ $($resp. $(ii))$.   There exists a family
  $\{f_a^t\}\in   C^{k+1/2,j}(\cL D_a)$ satisfying  \rta{cousin} $(i)$ $($resp. $(ii))$.
  \eth
  \begin{proof}
  In the proof of previous theorem, we have $\tilde g^t=\db g_a^t$ and $\{\tilde g^t\}\in C^{\infty,j}(\cL D)$. When considering boundary, we can only claim  $\{\var_\nu^t\}\in C^{k,j}(\{D_{e_\nu}^t\})$ obtained by parameterizing $\{\ov{D^t}\}$ via embeddings $\{\Gaa^t\}\in C^{k+1,j}(\ov {\cL D})$. Thus, we have $\{\tilde g^t\} \in C^{k,j}(\ov{\cL D})$.
With $k-1\geq j$,   using \rt{spc}, we solve $\db u^t=\tilde g^t$ with $\{u^t\}\in C^{k+1/2,j}(\ov{\cL D})$.
  \end{proof}

When $n=1$, using \rt{spc1} we get a more precise result.
\th{cousin1d} Let $0<\all<1$. Let $k,j\in\ov\nn$ with $k\geq j$. Let $\{D^t\}$ be a smooth
family of bounded domains in $\cc$ with $C^{k+\all,j}\cap C^{1+\all,j}$ boundary. If  $\{f_{ab}^t\}$   in \rta{cousinb} are
in $\cL C^{k+\all,j}(\cL D_a\cap \cL D_b)$ for all $a,b$, then the solutions $\{f^t_a\}$ are
in $C^{k+\all,j}(\cL D_a)$.
\eth




\end{document}